\documentclass[11pt]{article}
\usepackage{latexsym, amssymb, enumerate, amsmath, amsthm,multicol}
\usepackage[utf8]{inputenc}
\usepackage{multicol}
\usepackage{fullpage}
\usepackage{array}
\usepackage{amsmath}
\usepackage{amssymb}
\usepackage{amsfonts}
\usepackage{multimedia}
\usepackage{float}
\usepackage{textcomp}
\usepackage{graphics, graphicx}
\usepackage{bbm}
\usepackage{mathtools}
\usepackage{enumerate}
\usepackage{framed} 
\usepackage{empheq}
\usepackage{fancyhdr}
\usepackage{wrapfig}
\usepackage{mdframed}
\usepackage{enumitem}
\usepackage{mathtools}
\usepackage{color}
\usepackage{epsfig}
\usepackage{caption}
\graphicspath{{./}}

\usepackage{multicol}
\renewcommand{\epsilon}{\varepsilon}

\let\phi\varphi

\newtheorem{thm}{Theorem}[section]

\newtheorem{lemma}[thm]{Lemma}

\newtheorem{proposition}[thm]{Proposition}

\newtheorem{theorem}[thm]{Theorem}

\usepackage{epstopdf}
\usepackage{algorithmic}
\ifpdf
  \DeclareGraphicsExtensions{.eps,.pdf,.png,.jpg}
\else
  \DeclareGraphicsExtensions{.eps}
\fi
\newtheorem{remark}[thm]{Remark}

\newcommand{\RR}{\mathbb{R}}

\newcommand{\N}{\mathbb{N}}

\renewcommand{\S}{\mathbb{S}}

\newcommand{\eps}{\varepsilon}

\newcommand{\rd}{\mathrm{d}}

\newcommand{\mC}{\mathcal C}
\newcommand{\mD}{\mathcal D}
\newcommand{\mB}{\mathcal B}
\newcommand{\bla}{\big\langle}
\newcommand{\bra}{\big\rangle}

\let\oldhat\hat
\renewcommand{\v}[1]{\mathbf{#1}}
\newcommand{\op}[1]{\mathsf{#1}}
\renewcommand{\hat}[1]{\oldhat{\mathbf{#1}}}

\newcommand{\TheTitle}{Exponential decay to equilibrium for a fibre lay-down process on a moving conveyor belt} 
\title{{\TheTitle}}%\thanks{This work was funded by the Fog Research Institute under contract no.~FRI-454.}}

\author{
  Emeric Bouin\thanks{CEREMADE - Universit\'e Paris-Dauphine, UMR CNRS 7534, Place du Mar\'echal de Lattre de Tassigny, 75775 Paris Cedex 16, France (bouin@ceremade.dauphine.fr).}
  \and
  Franca Hoffmann\thanks{CCA, Centre for Mathematical Sciences, University of Cambridge, Wilberforce Road, Cambridge CB3 0WA, UK
  (fkoh2@cam.ac.uk).}
  \and
  Cl\'ement Mouhot \thanks{DPMMS, Centre for Mathematical Sciences, University of Cambridge, Wilberforce Road, Cambridge CB3 0WA, UK (c.mouhot@dpmms.cam.ac.uk).}
}

\usepackage{amsopn}

\begin{document}

\maketitle

% REQUIRED
\begin{abstract}
  We show existence and uniqueness of a stationary state for a kinetic
  Fokker-Planck equation modelling the fibre lay-down process in the
  production of non-woven textiles. Following a micro-macro
  decomposition, we use hypocoercivity techniques to show exponential
  convergence to equilibrium with an explicit rate assuming the
  conveyor belt moves slow enough. This work is an extension of
  (Dolbeault et al., 2013), %\cite{Dolbeault2012short}
  where the authors consider the case of a stationary conveyor
  belt. Adding the movement of the belt, the global Gibbs state is not
  known explicitly. We thus derive a more general hypocoercivity
  estimate from which existence, uniqueness and exponential
  convergence can be derived. To treat the same class of potentials as
  in (Dolbeault et al., 2013), %\cite{Dolbeault2012short}
  we make use of an additional weight function following the Lyapunov
  functional approach in (Kolb et al., 2013).% \cite{Kolb2013}.
\end{abstract}
\noindent{\bf AMS Class. No:} {35B20, 35B40, 35B45, 35Q84}
\section{Introduction}

The mathematical analysis of the fibre lay-down process in the
production of non-woven textiles has seen a lot of interest in recent
years \cite{Marheineke2006, Marheineke2007, Goetz2007,
  KlarMaringerWegener2010,Kolb2011, Dolbeault2012short,
  Kolb2013}. Non-woven materials are produced in melt-spinning
operations: hundreds of individual endless fibres are obtained by
continuous extrusion through nozzles of a melted polymer. The nozzles
are densely and equidistantly placed in a row at a spinning beam. The
visco-elastic, slender and in-extensible fibres lay down on a moving
conveyor belt to form a web, where they solidify due to cooling air
streams. Before touching the conveyor belt, the fibres become
entangled and form loops due to highly turbulent air flow. In
\cite{Marheineke2006} a general mathematical model for the fibre
dynamics is presented which enables the full simulation of the
process. Due to the huge amount of physical details, these simulations
of the fibre spinning and lay-down usually require an extremely large
computational effort and high memory storage, see
\cite{Marheineke2007}. Thus, a simplified two-dimensional stochastic
model for the fibre lay-down process, together with its kinetic limit,
is introduced in \cite{Goetz2007}. Generalisations of the
two-dimensional stochastic model \cite{Goetz2007} to three dimensions
have been developed by Klar et al. in \cite{KlarMaringerWegener2010}
and to any dimension $d \geq 2$ by Grothaus et al. in
\cite{GrothausKlar}.

We now describe the model we are interested in, which comes from
\cite{Goetz2007}. We track the position $x(t) \in \RR^2$ and the angle
$\alpha(t) \in \S^1$ of the fibre at the lay-down point where it
touches the conveyor belt. Interactions of neighbouring fibres are
neglected. If $x_0(t)$ is the lay-down point in the coordinate system
following the conveyor belt, then the tangent vector of the fibre is
denoted by $\tau(\alpha(t))$ with
$\tau(\alpha) = (\cos \alpha, \sin \alpha)$.  Since the extrusion of
fibres happens at a constant speed, and the fibres are in-extensible, the
lay-down process can be assumed to happen at constant normalised speed
$\| x_0'(t)\| =1$. If the conveyor belt moves with constant speed
$\kappa$ in direction $e_1 = (1,0)$, then
\begin{equation*}
 \frac{\rd x}{\rd t} =  \tau(\alpha) + \kappa e_1.
\end{equation*}
Note that the speed of the conveyor belt cannot exceed the lay-down
speed: $0 \leq \kappa \leq 1$. The fibre dynamics in the deposition
region close to the conveyor belt are dominated by the turbulent air
flow.  Applying this concept, the dynamics of the angle $\alpha(t)$
can be described by a deterministic force moving the lay-down point
towards the equilibrium $x=0$ and by a Brownian motion modelling the
effect of the turbulent air flow. We obtain the following stochastic
differential equation for the random variable $X_t = (x_t, \alpha_t)$
on $\RR^2 \times \S^1$,
\begin{equation} \label{stochastic}
 \begin{cases}
 \rd x_t &= \left( \tau(\alpha_t) + \kappa e_1 \right) \rd t,\\[0.2cm]
 \rd\alpha_t &= \left[- \tau^\bot (\alpha_t) \cdot \nabla_x V(x_t) \right] \, \rd t + A \, \rd W_t\, ,
\end{cases}
\end{equation}
where $W_t$ denotes a one-dimensional Wiener process, $A>0$ measures
its strength relative to the deterministic forcing,
$\tau^\bot=(-\sin\alpha,\cos\alpha)$, and
$V: \RR^2 \longrightarrow \RR$ is an external potential carrying
information on the coiling properties of the fibre.  More precisely,
since a curved fibre tends back to its starting point, the change of
the angle $\alpha$ is assumed to be proportional to
$\tau^\bot (\alpha) \cdot \nabla_x V(x)$.  It has been shown in
\cite{Kolb2013} that under suitable assumptions on the external
potential $V$, the fibre lay down process (\ref{stochastic}) has a
unique invariant distribution and is even geometrically ergodic (see
Remark~\ref{ergodic rmk}). The stochastic approach yields exponential
convergence in total variation norm, however without explicit rate. We
will show here that a stronger result can be obtained with a
functional analysis approach. Our argument uses crucially the
construction of an additional weight functional for the fibre lay-down
process in the case of unbounded potential gradients inspired by
\cite[Proposition 3.7]{Kolb2013}.

The probability density function $f(t, x, \alpha)$ corresponding to
the stochastic process (\ref{stochastic}) is governed by the
Fokker-Planck equation
\begin{equation} \label{kinetic1}
  \partial_t f + (\tau + \kappa e_1) \cdot \nabla_x f - \partial_\alpha \left( \tau^\bot \cdot \nabla_x V f \right) = D \partial_{\alpha \alpha} f
\end{equation}
with diffusivity $D = A^2/2$. We state below assumptions on the
external potential $V$ that will be used regularly throughout the
paper:
\begin{itemize}
\item[] \textbf{(H1)} \textit{Regularity and symmetry}:
  $V \in C^2(\RR^2)$ and $V$ is spherically symmetric outside some
  ball $B(0,R_V)$.
  \item[] \textbf{(H2)} \textit{Normalisation}: $\int_{\RR^2} e^{-V(x)} \,\rd x = 1$.
  \item[] \textbf{(H3)} \textit{Spectral gap condition} (Poincar\'e
    inequality): there exists a positive constant $\Lambda$ such that
    for any $u \in H^1 (e^{-V} \rd x) $ with
    $\int_{\RR^2} u e^{-V} \rd x = 0$,
	$$
	\int_{\RR^2} \left| \nabla_x u \right|^2 e^{-V}\, \rd x \geq \Lambda \int_{\RR^2} u^2 e^{-V} \,\rd x.
	$$
      \item[] \textbf{(H4)} \textit{Pointwise regularity condition on
          the potential}: there exists $c_1>0$ such that for any
        $x \in \RR^2$, the Hessian $\nabla_x^2 V$ of $V(x)$ satisfies
	$$ |\nabla_x^2 V(x)| \leq c_1 (1+|\nabla_x V(x)|).$$
  \item[] \textbf{(H5)} \textit{Behaviour at infinity}: 
  \begin{equation*}
 \lim_{|x|\to\infty} \frac{|\nabla_x V(x)|}{V(x)}=0, \qquad
 \lim_{|x|\to\infty} \frac{|\nabla^2_x V(x)|}{|\nabla_x V(x)|}=0\, .
\end{equation*}
\end{itemize}
\begin{remark} Assumptions \textbf{(H2-3-4)} are as stated
  in \cite{Dolbeault2012short}. Assumption \textbf{(H1)} assumes
  regularity of the potential that is stronger and included in that
  discussed in \cite{Dolbeault2012short} since \textbf{(H1)} implies
  $V \in W_{\text{loc}}^{2, \infty} (\RR^2)$.  Assumption
  \textbf{(H5)} is only necessary if the potential gradient
  $|\nabla_x V|$ is unbounded.  Both bounded and unbounded potential
  gradients may appear depending on the physical context, and we will
  treat these two cases separately where necessary.  A typical example
  for an external potential satisfying assumptions
  \textbf{(H1-2-3-4-5)} is given by
 \begin{equation}\label{ergodicV}
 V(x)=K\left(1+|x|^2\right)^{s/2}
 \end{equation}
 for some constants $K>0$ and $s\geq1$ \cite{Dolbeault2012long, Kolb2013}. The potential \eqref{ergodicV} satisfies \textbf{(H3)} since 
 $$\liminf_{|x|\to \infty} \left(|\nabla_x V|^2 -2 \Delta_x V\right) >0\, ,$$ 
 see for instance \cite[A.19. Some criteria for Poincar\'e inequalities, page 135]{Villani}. The other assumptions are trivially
 satisfied as can be checked by direct inspection. In this family of
 potentials, the gradient $\nabla_x V$ is bounded for $s=1$ and
 unbounded for $s>1$.
\end{remark}
\begin{remark}\label{ergodic rmk} The proof of ergodicity in
  \cite{Kolb2013} assumes that the potential satisfies
\begin{equation}\label{ergodicH}
 \lim_{|x|\to\infty} \frac{|\nabla_x V(x)|}{V(x)}=0\, , \qquad
 \lim_{|x|\to\infty} \frac{|\nabla^2_x V(x)|}{|\nabla_x V(x)|}=0\, ,  \qquad 
 \lim_{|x|\to\infty} \left|\nabla_xV(x)\right|=\infty\, .
\end{equation}
Under these assumptions, there exists an invariant distribution $\nu$ to the fibre lay-down process (\ref{stochastic}), and some constants $C(x_0)>0$, $\lambda>0$ such that
\begin{equation*}
 \left\|\mathcal{P}_{x_0,\alpha_0} \left(X_t \in \,\cdot\right)-\nu\right\|_{TV}\leq C(x_0)e^{-\lambda t}\, ,
\end{equation*}
where $\left\| \cdot \right\|_{TV}$ denotes the total variation norm.
The stochastic Lyapunov technique applied in \cite{Kolb2013} however
does not give any information on how the constant $C(x_0)$ depends on
the initial position $x_0$, or how the rate of convergence $\lambda$
depends on the conveyor belt speed $\kappa$, the potential $V$ and the
noise strength $A$. This can be achieved using hypocoercivity
techniques, proving convergence in a weighted $L^2$-norm, which is
slightly stronger than the convergence in total variation norm shown
in \cite{Kolb2013}. Conceptually, the conditions (\ref{ergodicH})
ensure that the potential $V$ is driving the process back inside a
compact set where the noise can be controlled. Our framework
\textbf{(H1-2-3-4-5)} is more general than conditions (\ref{ergodicH})
in some aspects (including bounded potential gradient) and more
restrictive in others (assuming a Poincar\'e inequality). The proof in
\cite{Kolb2013} relies on the strong Feller property which can be
translated in some cases into a spectral gap; it also uses
hypoellipticity to deduce the existence of a transition density, and
concludes via an explicit Lyapunov function argument. With our
framework~\textbf{(H1-2-3-4-5)}, and adapting the Lyapunov function
argument presented in \cite{Kolb2013} to control the effect of $\kappa
\partial_{x_1}$, we derive an explicit rate of convergence in terms of
$\kappa$, $D$ and $V$.
\end{remark}

To set up a functional framework, rewrite (\ref{kinetic1}) as
\begin{equation} \label{kinetic2}
 \partial_t f = \op{L}_\kappa f = \left( \op{Q} - \op{T}\right) f + \op{P}_\kappa f\, ,
\end{equation}
where the collision operator $\op{Q} := D \partial_{\alpha \alpha}$
acts as a multiplicator in the space variable $x$, $\op{P}_\kappa$ is
the perturbation introduced by the moving belt with respect to
\cite{Dolbeault2012short}:
$$
\op{P}_\kappa f := - \kappa e_1 \cdot \nabla_x f\, ,
$$
and the transport operator $\op{T}$ is given by
$$
\op{T}f := \tau \cdot \nabla_x f - \partial_\alpha \left(\tau^\bot \cdot \nabla_x V f \right).
$$
We consider solutions to \eqref{kinetic2} in the space $L^2(\rd \mu_\kappa):=L^2(\RR^2 \times \S^1, \rd \mu_\kappa)$ with measure
$$
\rd \mu_\kappa(x, \alpha) = \left(e^{V(x)} + \zeta\kappa g(x,\alpha)\right) \frac{\rd x \, \rd \alpha}{2 \pi}\, .
$$
We denote by $\langle \cdot, \cdot \rangle_\kappa$ the corresponding scalar product and by $\| \cdot \|_\kappa$ the associated norm. 
Here, $\zeta>0$ is a free parameter to be chosen later.
The construction of the weight $g$ depends on the boundedness of $\nabla_x V$. When it is bounded, no additional weight is needed to control the perturbation, and so we simply set $g \equiv 0$ in that case. When the gradient is unbounded, the weight is constructed thanks to the following proposition:
%%%%%%%%%%%%%%%%%%%%%%%%%
%%%%%%%%%%%%%%%%%%%%%%%%%
%%%%%%%%%%%%%%%%%%%%%%%%%
\begin{proposition}\label{Lyapunovprop}
Assume that $V$ satisfies $\textbf{(H1)}$ and \textbf{(H5)} and that 
\begin{equation*}  
\lim_{\vert x \vert \to\infty} \vert \nabla_x V \vert = + \infty.
\end{equation*} 
If $\kappa<1/3$ holds true, then there exists a function $g(x,\alpha)$, a constant $c=c(\kappa,D)>0$ and a finite radius $R=R(k,D,V)>0$ such that
\begin{equation}\label{Lyapunovcond}
 \forall \, |x| > R, \, \forall \, \alpha \in \S^1,
  \quad \mathcal{L}_\kappa(g)(x,\alpha) \leq - c\, |\nabla_x V(x)| g(x,\alpha)\, ,
\end{equation}
where $\mathcal{L}_\kappa$ is defined by
\begin{equation}\label{weightop}
  \mathcal{L}_\kappa(h) := D \partial_{\alpha \alpha} h
  +\left(\tau+\kappa e_1\right)\cdot \nabla_x h
  - \left( \tau^\bot \cdot \nabla_x V \right)  \partial_\alpha h - \left( \tau \cdot \nabla_x V \right) h \, .
\end{equation}
The weight $g$ is of the form
\begin{equation*}
g(x, \alpha):=\exp\left(\beta V(x) + |\nabla_x V (x)|\Gamma\left(\tau(\alpha)\cdot\frac{\nabla_x V(x)}{|\nabla_x V(x)|}\right)\right)\, ,
\end{equation*}
where the parameter $\beta>1$ and the function $\Gamma \in C^1\left([-1,1]\right)$, $\Gamma>0$ are determined along the proof and only depend on $\kappa$.
\end{proposition}
%%%%%%%%%%%%%%%%%%%%%%%%%%%
%%%%%%%%%%%%%%%%%%%%%%%%%%%
%%%%%%%%%%%%%%%%%%%%%%%%%%%
We show in Section~\ref{sec:weightg} the existence of such a weight
function $g$ under appropriate conditions following ideas from
\cite{Kolb2013}.

\medskip
We denote $\mC:=C_c^\infty\left(\RR^2 \times \S^1\right)$, and define the orthogonal projection $\op{\Pi}$ on the set of
local equilibria $\text{Ker} \, \op{Q}$%  consisting of all
% $\alpha$-independent distributions,
$$
\op{\Pi} f :=\int_{\S^1} f \,\frac{\rd \alpha}{2 \pi} \,,
$$
and the \textit{mass} $M_f$ of a given distribution $f \in L^2(\rd
\mu_\kappa)$,
\begin{equation*}
M_f :=\int_{\RR^2 \times \S^1}  f \,\frac{ \rd x \rd \alpha}{2 \pi}  \, .
\end{equation*}
Integrating \eqref{kinetic1} over $\RR^2 \times \S^1$ shows that the
mass of solutions of \eqref{kinetic1} is conserved over time, and
standard maximum principle arguments show that it remains non-negative
for non-negative initial data. 
The collision operator $\op{Q}$ is symmetric and satisfies
$$
\forall \, f \in \mC, \quad \langle \op{Q} f, f \rangle_0 = - D \left\| \partial_\alpha f \right\|_0^2 \leq 0 \,,
$$
i.e. $\op{Q}$ is dissipative in $L^2(\rd \mu_0)$. Further, we have
$\op{T \Pi} f = e^{-V} \tau \cdot \nabla_x u_f$ for $f \in \mC$, with
$u_f := e^V \op{\Pi}f$, which implies $\op{\Pi T \Pi} = 0$ on $\mC$.
Since the transport operator $\op{T}$ is skew-symmetric with respect
to $\langle \cdot \, , \, \cdot \rangle_0$, 
$$
\langle \op{L}_\kappa f, f \rangle_0 = \langle \op{Q}f, f \rangle_0 + \langle \op{P}_\kappa f, f \rangle_0
$$
for any $f$ in $\mC$. In the case $\kappa = 0$, if the entropy
dissipation $-\langle \op{Q}f, f \rangle_0$ was coercive with respect
to the norm $\| \cdot \|_0$, exponential decay to zero would follow as
$t \to \infty$. However, such a coercivity property cannot hold since
$\op{Q}$ vanishes on the set of local equilibria. Instead, Dolbeault
et al. \cite{Dolbeault2012long} applied a strategy called
\emph{hypocoercivity} (as theorised in \cite{Villani}) and developed
by several groups in the 2000s, see for instance
\cite{MR2034753,MR1946444,MR2043729,MR1787105,MR2116276}.
% developed by several groups (see for instance
% \cite{MR2034753,MR1946444,MR2043729,MR1787105,MR2116276}) and
% theorised by Villani in \cite{Villani}. 
The full hypocoercivity
analysis of the long time behaviour of solutions to this kinetic model
in the case of a stationary conveyor belt, $\kappa = 0$, is completed
in \cite{Dolbeault2012short}. For technical applications in the
production process of non-wovens, one is interested in a model
including the movement of the conveyor belt, and our aim is to extend
the results in \cite{Dolbeault2012short} to small $\kappa > 0$.

\medskip We follow the approach of hypocoercivity for linear kinetic
equations conserving mass developed in \cite{Dolbeault2012long}, with
several new difficulties. Considering the case
$\kappa=0$, $\op{Q}$ and $\op{T}$ are closed operators on
$L^2(\rd \mu_0)$ such that $\op{Q}-\op{T}$ generates the
$\mC_0$-semigroup $e^{(\op{Q}-\op{T})t}$ on $L^2(\rd \mu_0)$. When
$\kappa>0$, we use the additional weight function $g>0$ to control the
perturbative term $\op{P}_\kappa$ in the case of unbounded potential
gradients; and show the existence of a $\mC_0$-semigroup for
$\op{L}_\kappa= \op{Q}-\op{T}+\op{P}_\kappa$ (see Section
\ref{sec:semigroup}). 
Unless otherwise specified, all computations are performed on the
operator core $\mC$, and can be extended to $L^2(\rd \mu_\kappa)$ by density arguments.

When $\kappa=0$, the hypocoercivity result in
\cite{Dolbeault2012long,Dolbeault2012short} is based on:
\emph{microscopic coercivity}, which assumes that the restriction of
$\op{Q}$ to $(\text{Ker}\,\op{Q})^\bot $ is coercive, and
\emph{macroscopic coercivity}, which is a spectral gap-like inequality
for the operator obtained when taking a parabolic drift-diffusion
limit, in other words, the restriction of $\op{T}$ to
$\text{Ker} \, \op{Q}$ is coercive. The two properties are satisfied in the case of a stationary conveyor belt:
\begin{itemize}
\item The operator $\op{Q}$ is symmetric and the Poincar\'e inequality
  on $\S^1$,
\begin{equation*}
\frac{1}{2\pi}\int_{\S^1} |\partial_\alpha f|^2\, \rd \alpha 
\geq \frac{1}{2\pi}\int_{\S^1} \left( f - \frac{1}{2\pi}\int_{\S^1} f \,\rd \alpha \right)^2 \,\rd \alpha,
\end{equation*}
implies that $- \langle \op{Q}f, f \rangle_0 \geq D\left\|(1-\op{\Pi})f\right\|_0^2$. 
\item The operator $\op{T}$ is skew-symmetric and for any
  $h \in L^2(\rd \mu_0)$ such that $u_h =e^V \op{\Pi}h \in H^1(e^{-V}\rd x)$ and
  $\int_{\RR^2 \times \S^1} h \, \rd \mu_0 = 0$, \textbf{(H3)} implies
\begin{equation*}
\hspace{-0.5cm}
 \left\|\op{T \Pi} h \right\|_0^2 = 
\frac{1}{4 \pi} \int_{\RR^2 \times \S^1} e^{-V} |\nabla_x u_h |^2
                                             \, \rd x \, \rd \alpha 
  \geq \frac{\Lambda}{4 \pi} \int_{\RR^2 \times \S^1} e^{-V} u_h^2
         \, \rd x \, \rd \alpha = \frac{\Lambda}{2} \, \left\|\op{\Pi} h \right\|_0^2.
\end{equation*}
\end{itemize}
\medskip

In the case $\kappa = 0$, the unique global normalised equilibrium
distribution $F_0 = e^{-V}$ lies in the intersection of the null
spaces of $\op{T}$ and $\op{Q}$. When $\kappa >0$, $F_0$ is not in the
kernel of $\op{P}_\kappa$ and we are not able to find the global Gibbs
state of (\ref{kinetic2}) explicitly. 
However, the hypocoercivity theory is based on a priori estimates
\cite{Dolbeault2012long} that are, as we shall prove, to some extent
stable under perturbation. Our main result reads: 
%%%%%%%%%%%%%%%%%%%%%
%%%%%%%%%%%%%%%%%%%%%HYPOTHM
%%%%%%%%%%%%%%%%%%%%%
\begin{theorem}\label{hypothm}
  Let $f_{\mbox{{\scriptsize {\em in}}}} \in L^2(\rd \mu_\kappa)$ and let
  \textbf{(H1-2-3-4-5)} hold. For $0<\kappa< 1$ small enough
  (with a quantitative estimate) and $\zeta>0$ large enough (with a
  quantitative estimate), there exists a unique non-negative stationary state
  $F_\kappa \in L^2(\rd \mu_\kappa)$ with unit mass $M_{F_\kappa} =
  1$. In addition, for any solution $f$ of \eqref{kinetic1} in $L^2(\rd \mu_\kappa)$ with mass
  $M_f$ and subject to the initial condition $f(t=0)=f_{\mbox{{\scriptsize
        {\em in}}}}$, we have
  \begin{align}\label{expdecay}
    \left\|f(t,\cdot) - M_f F_\kappa\right\|_{\kappa} 
     \leq C \left\|f_{\mbox{{\scriptsize {\em in}}}}-M_f F_\kappa\right\|_{\kappa}   \, e^{-\lambda_\kappa t}\, , 
  \end{align}
where the rate of convergence $\lambda_\kappa >0$ depends only on $\kappa$, $D$ and $V$, and the constant $C>0$ depends only on $D$ and $V$.
\end{theorem}
%%%%%%%%%%%%%%%%%%%
%%%%%%%%%%%%%%%%%%%
%%%%%%%%%%%%%%%%%%%

In the case of a stationary conveyor belt $\kappa=0$ considered
in~\cite{Dolbeault2012short}, the stationary state is characterised by
the eigenpair $(\Lambda_0, F_0)$ with $\Lambda_0=0$, $F_0 =e^{-V}$,
and so $\text{Ker}\,\op{L}_0=\langle F_0 \rangle$.
% In~\cite{Dolbeault2012short}, the exponential rate of convergence
% $\lambda_0$ was estimated explicitly in terms of the diffusion
% coefficient $D$. 
This means that there is an isolated eigenvalue
$\Lambda_0=0$ and a spectral gap of size at least $[-\lambda_0, 0]$
with the rest of the spectrum $\Sigma(\op{L_0})$ to the left of
$-\lambda_0$ in the complex plane.  Adding the movement of the
conveyor belt, Theorem~\ref{hypothm} shows that
$\text{Ker}\,\op{L}_\kappa=\langle F_\kappa \rangle$ and the
exponential decay to equilibrium with rate $\lambda_\kappa$
corresponds to a spectral gap of size at least $[-\lambda_\kappa,
0]$. Further, it allows to recover an explicit expression for the rate of
convergence $\lambda_0$ for $\kappa=0$ (see Step~5 in
Section~\ref{propproof}).  In general, we are not able to compute the
stationary state $F_\kappa$ for $\kappa>0$ explicitly, but $F_\kappa$
converges to $F_0=e^{-V}$ weakly as $\kappa \to 0$ (see the discussion
in Section~\ref{sec:concluding-remarks}). Let us finally emphasize
that a specific contribution of our paper is to introduce \emph{two}
(and not one as in \cite{Dolbeault2012long,Dolbeault2012short})
modifications of the entropy: 1) we first modify the \emph{space
  itself} with the coercivity weight $g$, then 2) we change the norm
with an auxiliary operator following the hypocoercivity approach.  \medskip

The rest of the paper deals with the case $\kappa>0$ and is organised
as follows. In Section~\ref{sec:hypo}, we prove the main
hypocoercivity estimate. This allows us to establish the existence of
solutions to \eqref{kinetic1} using semigroup theory and to deduce the
existence and uniqueness of a steady state in Section~\ref{sec:exuni}
by a contraction argument. In Section \ref{sec:weightg}, we give a
detailed definition of the weight function $g$ that is needed for the
hypocoercivity estimate in Section~\ref{sec:hypo}.

%%%%%%%%%%%%%%%%%%%%%%%%%%%%%%%%%%%%%%%%%%%%%%%%%%%%%%%%%
%%%%%%%%%%%%%%%%%%%%%%%%%%%%%%%%%%%%%%%%%%%%%%%%%%%%%%%%
%%%%%%%%%%%%%%%%%%%%%%%%%%%%%%%%%%%%%%%%%%%%%%%%%%%%%%%%%%%
%%%%%%%%%%%%%%%%%%%%%%%%%%%%%%%%%%%%%%%%%%%%%%%%%%%%%%%%%
\section{Hypocoercivity estimate}\label{sec:hypo}
%%%%%%%%%%%%%%%%%%%%%%%%%%%%%%%%%%%%%%%%%%%%%%%%%%%%%%%%%
%%%%%%%%%%%%%%%%%%%%%%%%%%%%%%%%%%%%%%%%%%%%%%%%%%%%%%%%
%%%%%%%%%%%%%%%%%%%%%%%%%%%%%%%%%%%%%%%%%%%%%%%%%%%%%%%%
%%%%%%%%%%%%%%%%%%%%%%%%%%%%%%%%%%%%%%%%%%%%%%%%%%%%%%%%

Following \cite{Dolbeault2012long} we introduce the auxiliary operator
\begin{equation*}
\op{A} := (1+(\op{T}\op{\Pi})^*(\op{T}\op{\Pi}))^{-1}(\op{T} \op{\Pi})^*\, ,
\end{equation*}
and a modified entropy, i.e. a \emph{hypocoercivity functional}
$\op{G}$ on $L^2(\rd \mu_\kappa)$:
\begin{equation*}
\op{G}[f] := \frac{1}{2} \|f\|_\kappa^2 + \eps_1 \langle \op{A} f, f \rangle_0
\, , \quad f \in L^2(\rd \mu_\kappa)
\end{equation*}
for some suitably chosen $\eps_1 \in (0,1)$ to be determined later. It
follows from \cite{Dolbeault2012long} that
$|\langle \op{A} f, f \rangle_0| \le \| f \|^2_0$. Also, $\| f \|^2_0 \le \| f \|^2_\kappa$ by construction of $\mu_\kappa$, and hence $\op{G}[\cdot]$ is norm-equivalent to
$\| \cdot \|_\kappa^2$:
\begin{equation}\label{normequiv}
\forall \,  f \in L^2(\rd \mu_\kappa), \quad 
\left(\frac{1-\eps_1}{2}\right)\, \|f\|_{\kappa}^2 \, \leq \, \op{G}[f] \, \leq \, \left(\frac{1+\eps_1}{2}\right)\,  \|f\|_{\kappa}^2, 
\end{equation}

In this section, we prove the following hypocoercivity estimate:
%%%%%%%%%%%%%%%%%%%%%%%%%%%%%
%%%%%%%%%%%%%%%%%%%%%%%%%%%%%HYPOESTIMATE
%%%%%%%%%%%%%%%%%%%%%%%%%%%%%
\begin{proposition}\label{hypoestimate}
  Assume that hypothesis \textbf{(H1-2-3-4-5)} hold and that
  $0<\kappa < 1$ is small enough (with a quantitative estimate). Let
  $f_{\mbox{{\scriptsize {\em in}}}} \in L^2(\rd \mu_\kappa)$ and
  $f=f(t,x,\alpha)$ be a solution of \eqref{kinetic1} in
  $L^2(\rd \mu_\kappa)$ subject to the initial condition
  $f(t=0)=f_{\mbox{{\scriptsize {\em in}}}}$. Then $f$ satisfies the
  following Gr\"{o}nwall type estimate:
\begin{equation} \label{Gbound}
\frac{\rd}{\rd t} \op{G}[f(t,\cdot)] \leq - \gamma_1 \op{G}[f(t,\cdot)]  + \gamma_2 M_f^2,
\end{equation}
where $\gamma_1>0, \gamma_2>0$ are explicit constants only depending
on $\kappa$, $D$ and $V$.
\end{proposition}

Note that the estimate (\ref{Gbound}) is stronger than what is
required for the uniqueness of a global Gibbs state, and represents an
extension of the estimate given in \cite{Dolbeault2012short}. When
applied to the difference of two solutions with the same mass,
(\ref{Gbound}) gives an estimate on the exponential decay rate towards
equilibrium.

%%%%%%%%%%%%%%%%%%%%%
%%%%%%%%%%%%%%%%%%%%%
%%%%%%%%%%%%%%%%%%%%%
\subsection{Proof of Proposition \ref{hypoestimate}}
\label{propproof}
%%%%%%%%%%%%%%%%%%%%%
%%%%%%%%%%%%%%%%%%%%%
%%%%%%%%%%%%%%%%%%%%%

Differentiate in time $\op{G}[f]$ to get
$$
\frac{\rd}{\rd t} \op{G}[f] = \op{D}_0 [f] + \op{D}_1 [f]+\op{D}_2[f]+\op{D}_3[f]\, ,
$$
where the \emph{entropy dissipation functionals} $\op{D}_0$, $\op{D}_1$, $\op{D}_2$ and $\op{D}_3$ are given by
\begin{align*}
 \op{D}_0[f] &:= \left\langle \op{Q}f, f \right\rangle_0 
- \eps_1 \left\langle \op{A T \Pi}f,\op{ \Pi} f \right\rangle_0  - \eps_1 \left\langle \op{A T}(1- \op{ \Pi})f, \op{ \Pi}f \right\rangle_0\\
&\qquad + \eps_1 \left\langle \op{TA}f, (1- \op{ \Pi})f \right\rangle_0 + \eps_1 \left\langle \op{AQ}f, \op{ \Pi}f \right\rangle_0\, , \\%\label{D0}
\op{D}_1 [f] &:= \eps_1 \left\langle \op{A P}_\kappa f , \op{ \Pi}f \right\rangle_0 + \eps_1 \left\langle \op{P}_\kappa^* \op{A} f, \op{ \Pi}f \right\rangle_0 \, , \\
\op{D}_2[f] &:=\left\langle \op{P}_\kappa f, f \right\rangle_0\, , \\
\op{D}_3[f] &:= \kappa\zeta \int_{\RR^2 \times \S^1} \op{L}_\kappa(f) f g \,\frac{\rd x \, \rd \alpha }{2\pi}\, .
\end{align*}
Note that the term $\langle \op{LA}f, f \rangle_0$ vanishes since it
has been shown in \cite{Dolbeault2012long} that $\op{A}=\op{\Pi A}$
and hence $\op{A}f \in \text{Ker}\,\op{Q}$. Further,
$\langle \op{T}f, f \rangle=0$ since $\op{T}$ is skew-symmetric. We
estimate the entropy dissipation of the case $\kappa=0$ as in
\cite{Dolbeault2012short}:

\bigskip

\noindent\textbf{ \# Step 1: Estimation of $\op{D}_0 [f]$.}

\bigskip

We will show the boundedness of $\op{D}_0$, which is in fact the dissipation functional for a stationary conveyor belt. We thus recall without proof in the following lemma some results from \cite{Dolbeault2012short}. 
%%%%%%%%%%%%%%%%%%%%%%%
%%%%%%%%%%%%%%%%%%%%%%%
\begin{lemma}[Dolbeault et al. \cite{Dolbeault2012short}]\label{lem:estDol}
The following estimates hold:
\begin{equation*}
 \langle \op{Q}f, f \rangle_0 \leq -  \left\|(1-\op{\Pi})f \right\|_0^2\, ,\qquad 
\left\|\op{AT}(1-\op{\Pi})f\right\|_0 \leq C_V \left\|(1-\op{\Pi})f\right\|_0\, ,
\end{equation*}
\begin{equation*}
\left\|\op{AQ}f\right\|_0 \leq \frac{D}{2} \left\|(1-\op{\Pi})f\right\|_0\, , \qquad \left\|\op{TA}f\right\|_0 \leq \left\|(1-\op{\Pi})f \right\|_0\, .
\end{equation*}
\end{lemma}
%%%%%%%%%%%%%%%%%%%%%%
%%%%%%%%%%%%%%%%%%%%%%
In order to control the contribution $\left\langle \op{A T \Pi}f,\op{ \Pi} f \right\rangle_0$ in $\op{D}_0$, we note that
$$\op{A T \Pi}=\left(1+(\op{T \Pi})^*\op{T \Pi}\right)^{-1}(\op{T\Pi})^*\op{T \Pi}$$ 
shares its spectral decomposition with $(\op{T\Pi})^*\op{T\Pi}$, and by \emph{macroscopic coercivity}
$$
\langle (\op{T\Pi})^*\op{T \Pi} f,f \rangle_0
= \left\|\op{T \Pi} f\right\|_0^2
= \left\|\op{T\Pi} (f - M_f e^{-V})\right\|_0^2 
\geq \frac{\Lambda}{2} \left\|\op{\Pi}(f - M_f e^{-V})\right\|_0^2\, .
$$
Hence,
\begin{equation*}\label{boundATPi}
\langle \op{AT\Pi} f, f\rangle_0 \geq \frac{\Lambda/2}{1 + \Lambda/2} \left\|\op{\Pi}(f - M_f e^{-V})\right\|_0^2 \,.
\end{equation*}
Now, recalling Lemma \ref{lem:estDol} and using $\left\|\op{\Pi}(f - M_f e^{-V})\right\|_0^2 = \left\|\op{\Pi}f\right\|_0^2 - M_f^2$, we estimate
\begin{align*}
\op{D}_0[f] \leq &(\eps_1 - D) \left\|(1-\op{\Pi})f \right\|_0^2 + \eps_1 \lambda_2\left\|(1-\op{\Pi})f\right\|_0 \|\op{\Pi}f\|_0
- \eps_1 \gamma_2  \left(\left\|\op{\Pi}f\right\|_0^2 - M_f^2 \right)\, ,
\end{align*}
with $\lambda_2:=C_V + D/2>0$ and $\gamma_2:=  \frac{\Lambda/2}{1 + \Lambda/2} > 0$.
%%%%%%%%%%%%%%%%%%%%%%%%%%%%%%
%%%%%%%%%%%%%%%%%%%%%%%%%%%%%%
%%%%%%%%%%%%%%%%%%%%%%%%%%%%%%
%%%%%%%%%%%%%%%%%%%%%%%%%%%%%%

\bigskip

\noindent
\textbf{ \# Step 2: Estimation of $\op{D}_1 [f]$.}

\bigskip

We now turn to the entropy dissipation functional $\op{D}_1$, which we
will estimate using elliptic regularity.  Instead of bounding
$\op{AP}_\kappa$, we apply an elliptic regularity strategy to its
adjoint, as for $\op{AT}(1-\op{\Pi})$ in
\cite{Dolbeault2012short}. Let $f \in L^2(\rd \mu_0)$ and define
$h := \left(1 + (\op{T\Pi})^*\op{T\Pi} \right)^{-1}f$ so that
$u_h = e^V \op{\Pi}h$ satisfies
\begin{equation*}%\label{ellipticeqn0}
\op{\Pi}f 
= e^{-V}u_h + \op{\Pi}\op{T}^* \op{T} \left( e^{-V}u_h \right)
= e^{-V}u_h - \frac{1}{2}\nabla_x \cdot \left( e^{-V} \nabla_x u_h\right).
\end{equation*}
We have used here the fact that in the space $L^2(\rd \mu_0)$: 
\begin{equation*}
  \left\{
    \begin{array}{l}
      \op{T}  = \tau \cdot \nabla_x - \partial_\alpha \left[ \left( \tau^\bot \cdot
    \nabla_x V\right)  \right]\, , \\[3mm]
  \op{T}^*   = - \tau \cdot \nabla_x   + \left( \tau^\bot \cdot
    \nabla_x V \right) \partial_\alpha   - \left( \tau \cdot \nabla_x
    V \right)  .
    \end{array}
  \right.
\end{equation*}
Then
$$
\op{A}^*f = \op{T\Pi} h = e^{-V} \tau \cdot \nabla_x u_h\, ,
$$
and since the adjoint for $\langle \cdot ,\cdot \rangle_0$ of the
perturbation operator $\op{P}_\kappa$ is given by
\begin{equation*}
\op{P}_\kappa^* = - \op{P}_\kappa - \op{P}_\kappa V \,,
\end{equation*}
it follows that
\begin{align*}
 \left\|(\op{AP}_\kappa)^* f\right\|_0^2 
 &= \left\| \kappa \,  \tau \cdot \nabla_x (e_1 \cdot \nabla_x u_h ) e^{-V} \right\|_0^2\\
 &= \frac{\kappa^2}{2} \int_{\RR^2 \times \S^1} e^{-V} |\tau \cdot \nabla_x\left( e_1 \cdot \nabla_x u_h \right)|^2 \, \rd\mu_0\\
 &= \frac{\kappa^2}{2} \int_{\RR^2} e^{-V} |\nabla_x\left( e_1 \cdot \nabla_x u_h \right)|^2 \, \rd x\\
 &\leq \frac{\kappa^2}{2} \left\|\nabla^2_x u_h \right\|^2_{L^2(e^{-V}\, \rd x)}\\
 &\leq \frac{\kappa^2}{2} C_V^2 \left\|\op{\Pi}f\right\|_0^2\, ,
\end{align*}
where in the last inequality we have used an elliptic regularity
estimate. This estimate turns out to be a particular case of
\cite[Proposition 5 and Sections 2-3]{Dolbeault2012short}, where the
positive constant $C_V$ is the same as in Lemma \ref{lem:estDol} reproduced from \cite{Dolbeault2012short}.  This
concludes the boundedness of $\op{AP}_\kappa$,
\begin{equation} \label{boundAP} 
  \left\|\op{AP}_\kappa f \right\|_0 \leq
  \kappa \frac{C_V}{\sqrt{2}}\, \left\|\op{\Pi}f\right\|_0
  \leq \kappa \frac{C_V}{\sqrt{2}}\, \left\|f\right\|_0\, .
\end{equation}

%%%%%%%%%%%%%%%%%%%%%%%%%%%%%boundP^*A

Using a similar approach for the operator $\op{P}^*_\kappa \op{A}$, we
rewrite its adjoint as
$$
\op{A}^*\op{P}_\kappa f = \op{T \Pi} \tilde h\, ,
$$
where we define
$\tilde h := (1+(\op{T\Pi})^*\op{T\Pi})^{-1}\op{P}_\kappa f$ for a
given $f \in L^2(\rd \mu_0)$, or equivalently
$$
e^{-V}u_{\tilde h} - \frac{1}{2} \nabla_x \cdot \left(e^{-V}\nabla_x
  u_{\tilde h} \right) = \op{\Pi P}_\kappa f = \op{P}_\kappa \op{\Pi}
f\, .
$$
Multiplying by $u_{\tilde h}$ and integrating over $\RR^2$, we have
\begin{align*}
 \left\|u_{\tilde h}\right\|^2_{L^2 (e^{-V}\, \rd x)} + \frac{1}{2} \left\|\nabla_x u_{\tilde h} \right\|^2_{L^2 (e^{-V}\, \rd x)}
 &= -\kappa \int_{\RR^2} \v{e}_1 \cdot \nabla_x \left(\op{\Pi}f\right)u_{\tilde h} \,\rd x\\
 &= \kappa \int_{\RR^2}\left(\op{\Pi}f \right) \v{e}_1 \cdot \nabla_x u_{\tilde h} \,\rd x\\
 &\leq \kappa \int_{\RR^2} \left| \nabla_x u_{\tilde h} e^{-V/2} \right|  \left|\op{\Pi}f e^{V/2}\right| \, \rd x\\
 &\leq \kappa \left\|\nabla_x u_{\tilde h} \right\|_{L^2(e^{-V}\, \rd x)} \left\|\op{\Pi}f\right\|_0\\
 &\leq \frac{1}{4} \left\|\nabla_x u_{\tilde h} \right\|^2_{L^2(e^{-V}\, \rd x)} + \kappa^2 \left\|\op{\Pi}f\right\|_0^2\, .
\end{align*}
This inequality is a
$H^1(e^{-V} \,\rd x) \rightarrow H^{-1}(e^{-V}\, \rd x)$ elliptic
regularity result.  Hence,
\begin{equation*}
 \left\|\op{A}^*\op{P}_\kappa f\right\|_0^2 = \left\|\op{T \Pi} h\right\|_0^2 
 = \frac{1}{2} \left\|\nabla_x u_{\tilde h} \right\|^2_{L^2 (e^{-V}\, \rd x)}
 \leq 2 \kappa^2 \left\|\op{\Pi}f\right\|_0^2 \, ,
\end{equation*}
and so we conclude
\begin{equation}\label{boundP*A}
 \left\|\op{P}^*_\kappa\op{A} f\right\|_0 
%  =\left\|\op{P}^*_\kappa\op{A} (1-\op{\Pi})f\right\|_0
 \leq \sqrt{2} \kappa \left\|(1-\Pi)f\right\|_0 \le \sqrt{2} \kappa \left\|f\right\|_0 \, .
\end{equation}

%%%%%%%%%%%%%%%%%%%%%%%
%%%%%%%%%%%%%%%%%%%%%%%
Combining (\ref{boundAP}) and (\ref{boundP*A}), the entropy
dissipation functional $\op{D}_1$ is bounded by
\begin{align*}
  \op{D}_1 [f] 
  \leq \kappa \eps_1 \left(\frac{C_V}{\sqrt{2}} + \sqrt{2} \right)\, \left\|f \right\|_0^2
  = 2 \kappa \lambda_1 \, \|f\|_0^2, \notag
\end{align*}
where we defined $\lambda_1 := \frac{1}{2}\left(\frac{C_V}{\sqrt{2}} + \sqrt{2} \right)$.
%
%%%%%%%%%%%%%%%%%%%%%%%%%%%%%
%%%%%%%%%%%%%%%%%%%%%%%%%%%%%
%%%%%%%%%%%%%%%%%%%%%%%%%%%%%
\bigskip

\noindent
\textbf{\# Step 3: Estimation of $\op{D}_2[f]$.}

\bigskip

Using integration by parts, we have 
\begin{align*}
 \langle \op{P}_\kappa f, f \rangle_0 =
                                        \frac{\kappa}{2}
                                        \int_{\RR^2 \times \S^1}
                                        \left(e_1 \cdot \nabla_x
                                        V\right) f^2 e^V \,\frac{\rd x \, \rd \alpha}{2 \pi} \, .
\end{align*}
The estimation of this term goes differently depending on the boundedness of $\nabla_x V$. 

If $\nabla_x V$ is bounded, we write
\begin{equation*}
\op{D}_2 [f] \leq
\vert \langle \op{P}_\kappa f, f \rangle_0 \vert \leq \frac{\kappa}{2}  \Vert \nabla_x V \Vert_\infty \|f\|_0^2 = \frac{\kappa}{2}  \Vert \nabla_x V \Vert_\infty \|f\|_\kappa^2,
\end{equation*}
where we have used $\|f\|_\kappa = \|f\|_0$, since $g \equiv 0$.

Assume now that  $ \vert \nabla_x V \vert \to \infty$ as $\vert x \vert \to\infty$. Thanks to the choice of $g$, we have the estimate
\begin{equation}\label{gradVbound}
\op{D}_2 [f] \leq
\vert \langle \op{P}_\kappa f, f \rangle_0 \vert \leq \frac{\kappa}{2} \int_{\RR^2 \times \S^1} |\nabla_x V| f^2 e^V \,\frac{\rd x \, \rd \alpha}{2 \pi}
 \leq \frac{\kappa}{2}  C_3  \int_{\RR^2 \times \S^1} f^2 g \,\frac{\rd x \, \rd \alpha}{2 \pi},
\end{equation}
with
$$
C_3 := \sup_{x \in \RR^2} \left( |\nabla_x V| e^{V} g^{-1}\right),
$$
which is finite by \textbf{(H5)}.

\bigskip

\noindent
\textbf{\# Step 4: Estimation of $\op{D}_3[f]$.}

\bigskip

We start by recalling that this estimate is only relevant when $\nabla_x V$ is unbounded. Indeed, in the opposite case, $\op{D}_3[f] = 0$ since $g \equiv 0$ by definition. 
By the identity 
\begin{equation*}%\label{Lidentity}
 \int_{\RR^2\times\S^1} \op{L}_\kappa(f) f g\, \rd x \, \rd \alpha
= \frac{1}{2}\int_{\RR^2\times\S^1} \mathcal{L}_\kappa(g) f^2\, \rd x \,\rd \alpha
-D\int_{\RR^2\times\S^1} \left\vert \partial_\alpha f\right \vert^2 g\, \rd x \,\rd \alpha
\end{equation*}
with $\mathcal{L}_\kappa$ as defined in \eqref{weightop},
we have
\begin{align}
 \op{D}_3 [f] \leq  \kappa\zeta \left( \frac{1}{2}\int_{\RR^2\times\S^1} \mathcal{L}_\kappa(g) f^2\, \frac{\rd x \, \rd \alpha}{2\pi} 
\right)\label{D2unbounded}.
\end{align}
Proposition \ref{Lyapunovprop} allows us to control the $g$-weighted
$L^2$-norm outside some fixed ball. More precisely, take $R>0$ in
\eqref{Lyapunovcond} large enough s.t.  $|\nabla_x V|\geq 1$ for all
$|x|> R$, then
\begin{align}
&\int_{\RR^2\times\S^1} \mathcal{L}_\kappa(g) f^2\, \frac{\rd x \, \rd \alpha}{2\pi} \notag\\
&\qquad  \leq \int_{\S^1}\int_{|x|<R}
        \mathcal{L}_\kappa(g) f^2 \,\frac{\rd x \, \rd \alpha}{2\pi} -  
        c\int_{\S^1}\int_{|x|>R}
         \vert \nabla_x V \vert f^2 g\,\frac{\rd x \, \rd \alpha}{2\pi}\notag\\
&\qquad \leq \int_{\S^1}\int_{|x|<R} \left( \left(
        \mathcal{L}_\kappa(g) + c g \right) e^{-V}\right) f^2 e^V \,\frac{\rd x \, \rd \alpha}{2\pi} -c \int_{\RR^2 \times \S^1} f^2 g \, \frac{\rd x \, \rd \alpha}{2\pi}\notag\\
&\qquad  \leq  C_4(R) \Vert f \Vert_0^2 - c\int_{\RR^2 \times \S^1}f^2 g \, \frac{\rd x \, \rd \alpha}{2\pi}\, ,\label{dtgbound}
\end{align}
where
$C_4(R):=\sup_{\vert x \vert \leq R}\left( \vert
  \mathcal{L}_\kappa(g) + cg  \vert e^{-V} \right) $.

\begin{remark}
  Observe here that one could take advantage of the growth of
  $\nabla_x V$ 
%   (recall we are in the case where it goes to infinity at
%   infinity) 
  by playing with the cut-off parameter $R$ and keeping track
  of $\min_{|x| \ge R} |\nabla_x V|$ in the negative term. It could
  lead to more optimal constants but we chose instead to vary the
  parameter $\zeta$ in front of the coercivity weight $g$ in the
  measure $\mu_\kappa$ for simplicity.
\end{remark}

%%%%%%%%%%%%%%%%%%%%%%%%%
%%%%%%%%%%%%%%%%%%%%%%%%%
%%%%%%%%%%%%%%%%%%%%%%%%%
%%%%%%%%%%%%%%%%%%%%%%%%%
%%%%%%%%%%%%%%%%%%%%%%%%%

\bigskip

\noindent
\textbf{\# Step 5: Putting the four previous steps together.}

\bigskip
Combine the previous steps into
\begin{align*}
\op{D}_0[f] + \op{D}_1 [f] 
\leq &(\eps_1 - D) \left\|(1-\op{\Pi})f \right\|_0^2 + \eps_1
       \lambda_2 \left\|(1-\op{\Pi})f\right\|_0
       \|\op{\Pi}f\|_0 \\&- \eps_1 \gamma_2  \left(\left\|\op{\Pi}f\right\|_0^2 - M_f^2 \right) 
    +2 \kappa \lambda_1 \, \|f\|_0^2\\
= &- (D - \eps_1-2\kappa \lambda_1) \left\|(1-\op{\Pi})f \right\|_0^2
    + \eps_1 \lambda_2
    \left\|(1-\op{\Pi})f\right\|_0 \|\op{\Pi}f\|_0 \\&- \left( \eps_1 \gamma_2 - 2 \kappa \lambda_1 \right) \left\|\op{\Pi}f\right\|_0^2 
+ \eps_1 \gamma_2 M_f^2 \\
\leq &- \left(D - \eps_1-2\kappa \lambda_1- \frac{\eps_1 \lambda_2 b}{2} \right) \left\|(1-\op{\Pi})f \right\|_0^2\\
%     + \eps_1 \lambda_2
%     \left\|(1-\op{\Pi})f\right\|_0 \|\op{\Pi}f\|_0 \\
    &- \left( \eps_1 \gamma_2 - 2 \kappa \lambda_1 -\frac{\eps_1 \lambda_2}{2b} \right) \left\|\op{\Pi}f\right\|_0^2 
+ \eps_1 \gamma_2 M_f^2 \\
\leq &- 2\xi(\kappa)  \| f\|_0^2 + \eps_1 \gamma_2 M_f^2\,, 
\end{align*}
by Young's inequality with the choice $b=\lambda_2/\gamma_2$, and 
where we used the fact that $\|(1-\op{\Pi})f\|_0^2+\|\op{\Pi}f\|_0^2=\|f\|_0^2$. Here, $\xi(\kappa)$ is explicit, and given by
\begin{align*}
\xi(\kappa) := &\frac12 \min \left\{ D - \eps_1\left(1+\frac{\lambda_2^2}{2 \gamma_2}\right) \ , \  \frac{\eps_1 \gamma_2}2 \right\}
                       -\kappa\lambda_1 \\
 = & \frac{D\gamma_2^2}{2\left(\gamma_2^2+2\gamma_2+\lambda_2^2\right)} -\kappa\lambda_1\, ,
\end{align*}
since the minimum in the first term is realised when the two arguments
are equal, fixing
$\eps_1=2D\gamma_2/\left(\gamma_2^2+2\gamma_2+\lambda_2^2\right)$. Note
that this choice of $\eps_1$ satisfies $\eps_1<D$ and
$\eps_1<1$. Choosing $\kappa$ small enough ensures $\xi(\kappa)>0$.
% Ensuring $\xi(\kappa)>0$ imposes the
% restriction 
% \begin{equation*}
% 0 \leq \kappa < \frac{D\gamma_2^2}{2\lambda_1\left(\gamma_2^2+2\gamma_2+\lambda_2^2\right)}\, .
% \end{equation*}
% % which is non-empty if the right-hand side is positive, satisfied if
% % $\eps_1$ small enough e.g. 
% % \begin{equation}\label{eps1restriction}
% % \eps_1 < \frac{2 \gamma_2 D}{\lambda_2^2} < D.
% % % \frac{\left(\frac{\Lambda}{2+\Lambda}\right)D}{\left(\frac{\Lambda}{2+\Lambda}\right) + \left( C_V + \frac{D}{2}\right)^2}
% % % < D\, .
% % \end{equation}
From this analysis we conclude
\begin{align} \label{D1D2}
\op{D}_0[f] + \op{D}_1[f]
\leq - 2 \xi(\kappa) \|f\|_0^2 + \eps_1 \gamma_2 M_f^2\, .
\end{align}

Let us now add the control of $\op{D}_2 + \op{D}_3$. 
If $\nabla_x V$ is bounded, $g \equiv 0$ and $\op{D}_3=0$: 
\begin{align*}
 \frac{\rd}{\rd t} \op{G}[f] &= \op{D}_0[f] + \op{D}_1[f] + \op{D}_2[f] \\
 &\leq  - \left(4 \xi(\kappa) - \kappa   \Vert \nabla_x V
   \Vert_\infty\right) \frac12 \|f\|_\kappa^2 + \eps_1 \gamma_2 M_f^2 \\
&\leq - \gamma_1 \op{G}[f] + \eps_1 \gamma_2 M_f^2
\end{align*}
by the norm equivalence (\ref{normequiv}). Here, we defined
\begin{equation*}
 \gamma_1:= \frac{4\xi(\kappa) - \kappa   \Vert \nabla_x V \Vert_\infty}{1+\eps_1} >0\,.
\end{equation*}
% For our analysis to work, we have to impose that the movement of the
% conveyor belt is slow enough with respect to the speed at which the
% fibres are produced, see Remark \ref{rmk:krestriction}.

When $\nabla_x V$ is unbounded,  \eqref{gradVbound}-\eqref{D2unbounded}-\eqref{dtgbound}-\eqref{D1D2} imply
\begin{align*}
\frac{\rd}{\rd t} \op{G}[f] = &\op{D}_0[f] + \op{D}_1[f] + \op{D}_2[f] + \op{D}_3[f]\notag\\
\leq & 
- 2\xi(\kappa)  \| f\|_0^2 +\eps_1 \gamma_2 M_f^2 
+ \frac{\kappa}{2}  C_3  \int_{\RR^2 \times \S^1} f^2 g \,\frac{\rd x \, \rd \alpha}{2 \pi}\\
&\quad +\frac{\kappa\zeta}{2} \left( C_4(R) \Vert f \Vert_0^2 - c\int_{\RR^2 \times \S^1}f^2 g \, \frac{\rd x \, \rd \alpha}{2\pi} \right)\\
= &
- \frac12 \left(4 \xi(\kappa) - \kappa\zeta
       C_4(R)\right)\|f\|_0^2 - \frac{\kappa\zeta}{2} \left(c - \frac{
       C_3}{\zeta} \right) \int_{\RR^2 \times \S^1} f^2 g\,\frac{\rd
    x \, \rd \alpha}{2\pi} + \eps_1 \gamma_2 M_f^2 \notag\\
\leq &- \frac12 \min\left\{ 4 \xi(\kappa) - \kappa\zeta
       C_4(R),c - \frac{
       C_3}{\zeta}\right\} \Vert f \Vert_\kappa^2 + \eps_1 \gamma_2 M_f^2\notag \\       
\leq &- \gamma_1\, \op{G}[f]  + \eps_1 \gamma_2 M_f^2
\end{align*}
again by norm equivalence \eqref{normequiv}, and where we defined
\begin{align*}
  &\gamma_1:= \frac{1}{1+\eps_1}\min\left\{ 4 \xi(\kappa) - \kappa\zeta C_4(R),c - \frac{C_3}{\zeta}\right\}>0\,.
\end{align*}
This requires $\zeta>0$ to be large enough, and the upper bound
for $\kappa$ should be chosen accordingly:
\begin{equation*}
\zeta > \frac{C_3}{c}\, ,
\qquad
4 \xi(\kappa) - \kappa\zeta
       C_4(R) > 0\, .
\end{equation*}
In order to maximise the rate of convergence to equilibrium given $\kappa$, $D$ and $V$, one
can optimise $\gamma_1$ over $\zeta$ whilst respecting the above constraints.

\begin{remark}\label{rmk:krestriction}
  The condition $\gamma_1>0$ translates into an explicit upper bound
  on $\kappa$. More precisely, we require
  $\xi(\kappa)>\kappa u/4$ where
  $u:=\Vert \nabla_x V \Vert_\infty$ in the case of a bounded
  potential gradient, and $u:=\zeta C_4(R)$ otherwise.  This
  condition is satisfied for small enough $\kappa$:
 \begin{equation*}
  0\leq \kappa< \frac{\eps_1\gamma_2}{(4\lambda_1+u)}
  =\frac{2D\gamma_2^2}{(4\lambda_1+u)(\gamma_2^2+2\gamma_2+\lambda_2^2)}
 \end{equation*}
which also implies $\xi(\kappa)>0$. Recall that Proposition
 \ref{Lyapunovprop} requires $\kappa<1/3$ in the case of unbounded
 potential gradients. These conditions provide a range of $\kappa$
 for which Proposition \ref{hypoestimate} holds.  
\end{remark}

%%%%%%%%%%%%%%%%%%%%%%%%%%%%%%%%%%%%%%%%%%%%%%%%%%%%%%%%%
%%%%%%%%%%%%%%%%%%%%%%%%%%%%%%%%%%%%%%%%%%%%%%%%%%%%%%%%
%%%%%%%%%%%%%%%%%%%%%%%%%%%%%%%%%%%%%%%%%%%%%%%%%%%%%%%%%%%
%%%%%%%%%%%%%%%%%%%%%%%%%%%%%%%%%%%%%%%%%%%%%%%%%%%%%%%%%
\section{The coercivity weight $g$}\label{sec:weightg}
%%%%%%%%%%%%%%%%%%%%%%%%%%%%%%%%%%%%%%%%%%%%%%%%%%%%%%%%%
%%%%%%%%%%%%%%%%%%%%%%%%%%%%%%%%%%%%%%%%%%%%%%%%%%%%%%%%
%%%%%%%%%%%%%%%%%%%%%%%%%%%%%%%%%%%%%%%%%%%%%%%%%%%%%%%%
%%%%%%%%%%%%%%%%%%%%%%%%%%%%%%%%%%%%%%%%%%%%%%%%%%%%%%%%

In this section, we define the function $g$ in such a way that it allows us to control the loss of weight in the perturbation operator $\op{P}_\kappa$. When $\nabla_x V$ is bounded, we do not need any extra weight since then we may control the perturbation thanks to the stationary weight $e^V$, and so we set $g\equiv 0$ in that case. When it is not, Proposition \ref{Lyapunovprop} provides a suitable weight function $g$ by constructive methods.

\subsection{Proof of Proposition \ref{Lyapunovprop}}
%%%%%%%%%%%%%%%%%%
%%%%%%%%%%%%%%%%%%
%%%%%%%%%%%%%%%%%%
The proof is strongly inspired from \cite{Kolb2013}, however our weight is different since we work in an $L^2$-framework rather than in an $L^1$ one. 
%%%%%%%%%%%%begin of proof
Assuming $\nabla_x V$ is unbounded, we seek a weight $g$ of the form
\begin{equation*}
g(x, \alpha)=\exp\left(\beta V(x) + |\nabla_x V (x)|\Gamma\left(\tau(\alpha)\cdot\frac{\nabla_x V(x)}{|\nabla_x V(x)|}\right)\right)\, ,
\end{equation*}
where the parameter $\beta>1$ and the function $\Gamma \in C^1\left([-1,1]\right)$, $\Gamma>0$ are to be determined. We define
$$
Y(x,\alpha) :=  \tau(\alpha) \cdot \dfrac{\nabla_x V(x)}{\vert \nabla_x V(x) \vert}\, , \qquad 
Y^\bot (x,\alpha):=  \tau^\bot(\alpha) \cdot \dfrac{\nabla_x V(x)}{\vert \nabla_x V(x) \vert}\, ,
$$ 
and split the proof into four steps: 1) we rewrite statement
\eqref{Lyapunovcond} using the explicit expression of the weight $g$,
2) we simplify the obtained expression using assumption \textbf{(H5)},
3) we prove the equivalent statement obtained in Step 2 by defining a
suitable choice of $\Gamma(\cdot)$ and $\beta$, and 4) we demonstrate that
it is indeed possible to choose suitable parameters for the
calculations in Step~3 to hold, fixing explicit expressions where
possible.

\bigskip

\noindent
\textbf{\# Step 1: Rewriting the weight estimate \eqref{Lyapunovcond}.}

\bigskip

Applying the operator $\mathcal{L}_\kappa$ defined in \eqref{weightop} to $g$, we can compute explicitly
\begin{align*}
 \frac{\mathcal{L}_\kappa(g)}{g}
 =  &D\left( |\nabla_x V|\partial_{\alpha \alpha} \Gamma(Y) + |\nabla_x V|^2\vert \partial_\alpha \Gamma(Y) \vert^2 \right)\\
 &+ \left(\tau(\alpha)+\kappa e_1\right) \cdot \left( \beta \nabla_x V + \nabla_x \left( |\nabla_x V|\Gamma(Y) \right) \right) \\
 &- \vert \nabla_x V \vert^2 Y^\bot  \partial_\alpha \Gamma(Y) - \vert \nabla_x V \vert Y\, .
 \end{align*}
 Since
 \begin{equation*}
\partial_{\alpha} \Gamma = Y^\bot \Gamma'(Y)
\quad \text{and} \quad
\partial_{\alpha \alpha} \Gamma 
=  \partial_\alpha \left( Y^\bot \Gamma'(Y) \right) 
=   - Y \Gamma'(Y) + \vert Y^\bot \vert^2 \Gamma''(Y)\, ,
\end{equation*}
we get 
 \begin{align*}
 \frac{\mathcal{L}_\kappa(g)}{g} =&  D \left( \vert \nabla_x V \vert \left( - Y \Gamma'(Y) + \vert Y^\bot \vert^2 \Gamma''(Y) \right) + \vert \nabla_x V \vert^2 \vert Y^\bot \vert^2 \left(  \Gamma'(Y) \right)^2 \right) \\
 &+\left(\tau(\alpha)+\kappa e_1\right) \cdot \left( \beta \nabla_x V + \nabla_x \left( |\nabla_xV|\Gamma(Y) \right) \right)\\
 &- \vert \nabla_x V \vert^2 \vert Y^\bot \vert^2  \Gamma'(Y) - \vert \nabla_x V \vert Y\\
 =& (\beta - 1 -  D\Gamma'(Y) ) \vert \nabla_x V \vert Y + \kappa \beta e_1 \cdot \nabla_x V + \left(\tau(\alpha)+\kappa e_1\right) \cdot \nabla_x \left( |\nabla_xV|\Gamma(Y) \right) \\
 & +\vert Y^\bot \vert^2 \left(  D \vert \nabla_x V \vert
   \Gamma''(Y)  + \vert \nabla_x V \vert^2  \left[ D \left(  \Gamma'(Y) \right)^2 -  \Gamma'(Y)  \right] \right)\, .
\end{align*}
In order to see which $\Gamma$ to choose, let us divide by $|\nabla_x V|$ and denote the diffusion and transport part by
$$
\text{diff}(x, \alpha):= \left(\tau(\alpha)+\kappa e_1\right) \cdot \frac{\nabla_x \left( |\nabla_x V| \Gamma(Y) \right)}{|\nabla_x V|}\, , \qquad
\text{tran}(x):= \frac{e_1\cdot \nabla_x V}{|\nabla_x V|}\, .
$$
Now, we can rewrite the statement of Proposition \ref{Lyapunovprop}:
we seek a positive constant $c>0$ and a radius $R>0$ such that for
any $\alpha \in \S^1$ and $|x|>R$,
\begin{align*}
  (\beta - 1 - D\Gamma'(Y) ) &Y + \kappa \beta \text{tran}(x) +
  \text{diff}(x, \alpha) \\
  &+\vert Y^\bot \vert^2 \left( D\Gamma''(Y) +
    \vert \nabla_x V \vert \left[ D\left( \Gamma'(Y) \right)^2 -
      \Gamma'(Y) \right] \right) \leq -c\, .
\end{align*}
To achieve this bound, note that $|Y| \leq 1$ and
$|\text{tran}|\leq 1$ for all $(x, \alpha) \in \RR^2 \times
\S^1$.

\bigskip

\noindent
\textbf{\# Step 2: Simplifying the weight estimate.}

\bigskip

Further, the diffusion term $\text{diff}(\cdot)$ can be made
arbitrarily small outside a sufficiently large ball. Indeed, 
\begin{align*}
\text{diff}(x,\alpha)
  & =\left(\tau+\kappa e_1\right)\cdot \left[ \Gamma'\left( Y \right) \nabla_x  Y
    + \Gamma\left(Y\right)\dfrac{\nabla_x(\vert \nabla_x V \vert)}{\vert \nabla_x V \vert}\right]\, ,
\end{align*}
and both $|\nabla_x Y|$ and
$|\nabla_x(\vert \nabla_x V \vert)|/ |\nabla_x V|$ converge to zero as
$|x| \to \infty$ by assumption \textbf{(H5)}, and $\Gamma$ is
bounded. In other words, using the fact that the potential gradient is
unbounded, it remains to show that we can find constants
$\gamma>\kappa\beta>0$ and a radius $r_1>0$ such that
\begin{equation} \label{boundgamma}
\forall |x|>r_1, \quad
(\beta - 1 - D \Gamma' ) Y
+\vert Y^\bot \vert^2 \left( D\Gamma'' + \vert \nabla_x V \vert  \left[D \left(  \Gamma' \right)^2 -  \Gamma'  \right] \right)  \leq - \gamma\, .
\end{equation}
Then we can choose $r_2>0$ such that
$$
|x|>r_2 \quad \Longrightarrow \quad \forall \alpha \in \S^1, \quad \text{diff}(x,\alpha) \leq \frac{\gamma-\kappa\beta}{2} \, ,
$$
and we conclude for the statement of Proposition \ref{Lyapunovprop}
with $R:=\max\{r_1,r_2\}$ and $c:=(\gamma-\kappa\beta)/2>0$.

\bigskip

\noindent
\textbf{\# Step 3: Proof of the weight estimate.}

\bigskip

Proving \eqref{boundgamma} can be done by an explicit construction. We
define $\Gamma'\in C^0([-1,1])$ piecewise,
\begin{multicols}{2}
\vfill
 \begin{equation*}
 \Gamma'(Y)=
 \begin{cases}
  \delta^+	&\mbox{if } Y> \eps_0\, , \\[2mm]
  \frac{\delta^+ - \delta^-}{2 \eps_0} \left( Y + \eps_0 \right) + \delta^-
		&\mbox{if } |Y|\leq \eps_0\, , \\[2mm]
   \delta^-	&\mbox{if } Y<-\eps_0\, ,
 \end{cases}
\end{equation*}
\vfill\null
\columnbreak
\hspace{0.7cm}
\begin{center}
\mbox{} \hspace{0.2cm}
% \vspace{-1cm}
\def\svgwidth{140pt}
 %% Creator: Inkscape inkscape 0.48.4, www.inkscape.org
%% Accompanies image file 'gammaprimeplot.pdf' (pdf, eps, ps)
\begingroup%
  \makeatletter%
  \providecommand\color[2][]{%
    \errmessage{(Inkscape) Color is used for the text in Inkscape, but the package 'color.sty' is not loaded}%
    \renewcommand\color[2][]{}%
  }%
  \providecommand\transparent[1]{%
    \errmessage{(Inkscape) Transparency is used (non-zero) for the text in Inkscape, but the package 'transparent.sty' is not loaded}%
    \renewcommand\transparent[1]{}%
  }%
  \providecommand\rotatebox[2]{#2}%
  \ifx\svgwidth\undefined%
    \setlength{\unitlength}{336.35bp}%
    \ifx\svgscale\undefined%
      \relax%
    \else%
      \setlength{\unitlength}{\unitlength * \real{\svgscale}}%
    \fi%
  \else%
    \setlength{\unitlength}{\svgwidth}%
  \fi%
  \global\let\svgwidth\undefined%
  \global\let\svgscale\undefined%
  \makeatother%
  \begin{picture}(1,0.73473967)%
    \put(0,0){\includegraphics[width=\unitlength]{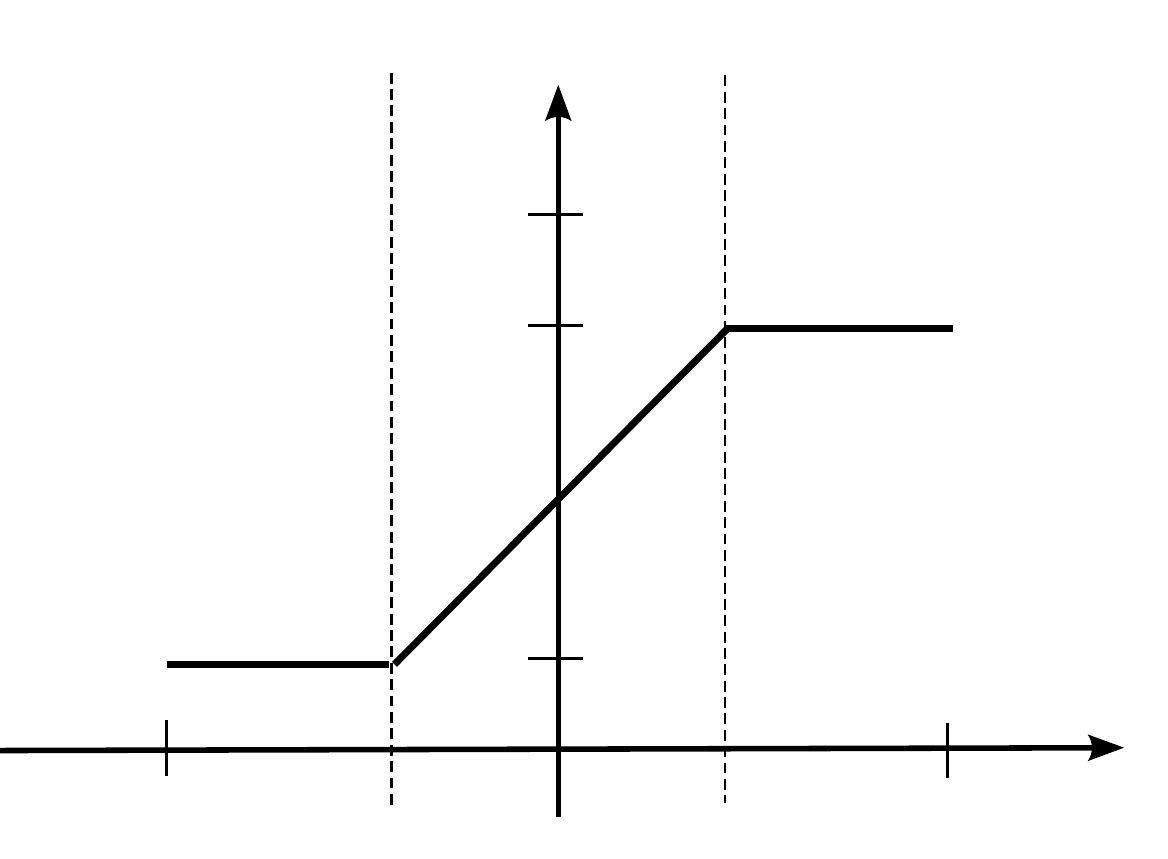}}%
    \put(0.14301509,0.02786811){\color[rgb]{0,0,0}\makebox(0,0)[b]{\smash{-1}}}%
    \put(0.80810772,0.02548963){\color[rgb]{0,0,0}\makebox(0,0)[b]{\smash{+1}}}%
    \put(0.34084194,0.01121878){\color[rgb]{0,0,0}\makebox(0,0)[b]{\smash{$-\varepsilon_0$}}}%
    \put(0.63523395,0.01121878){\color[rgb]{0,0,0}\makebox(0,0)[b]{\smash{$\varepsilon_0$}}}%
    \put(0.50661513,0.53686171){\color[rgb]{0,0,0}\makebox(0,0)[lb]{\smash{$\frac{1}{D}$}}}%
    \put(0.50789496,0.44172272){\color[rgb]{0,0,0}\makebox(0,0)[lb]{\smash{$\delta^+$}}}%
    \put(0.50661513,0.1610627){\color[rgb]{0,0,0}\makebox(0,0)[lb]{\smash{$\delta^-$}}}%
    \put(0.42964685,0.698598){\color[rgb]{0,0,0}\makebox(0,0)[b]{\smash{$\Gamma'(Y)$}}}%
    \put(0.98539003,0.03976048){\color[rgb]{0,0,0}\makebox(0,0)[b]{\smash{Y}}}%
  \end{picture}%
\endgroup%

 \vspace{-0.2cm}
 \captionof{figure}{Derivative of $\Gamma$\vspace{-0.4cm}}
\label{fighypo-intro}
\end{center}
\end{multicols}
% \begin{equation*}
%  \Gamma'(Y)=
%  \begin{cases}
%   \delta^+	&\mbox{if } Y> \eps_0, \\[2mm]
%   \frac{\delta^+ - \delta^-}{2 \eps_0} \left( Y + \eps_0 \right) + \delta^-
% 		&\mbox{if } |Y|\leq \eps_0, \\[2mm]
%    \delta^-	&\mbox{if } Y<-\eps_0,
%  \end{cases}
% \end{equation*}
% \begin{center}
% \def\svgwidth{250pt}
%  \input{gammaprimeplot.pdf_tex}
%  \captionof{figure}{Derivative of $\Gamma$}
% \label{fig}
% \end{center}
\noindent
where $0<\delta^-<\delta^+<1/D$ and $\eps_0 \in (0,1)$ are to be
determined. With this choice of $\Gamma'$, we can ensure that $\Gamma$
is strictly positive in the interval $[-1,1]$. Now, let us show that
there exist suitable choices of $\gamma$ and $\beta$ for the bound
(\ref{boundgamma}) to hold. More precisely, we choose a suitable
$\beta$ such that $(\beta-1)/D \in (\delta^-, \delta^+)$ and
$0<\gamma < \tilde \gamma$, defining
$\gamma:=\eps_0\left(1+D\delta^+-\beta\right)$ and
$\tilde\gamma:=\eps_0\left(\beta - 1-D\delta^-\right)$. We split our
analysis into cases:
\begin{itemize}
\item Assume $Y > \eps_0$. Then the LHS of (\ref{boundgamma}) can be bounded as follows:
\begin{equation*}
( \beta -1 - D\delta^+ ) Y + \delta^+ \left(D \delta^+ - 1\right) \vert \nabla_x V \vert \vert Y^\bot \vert^2   <  ( \beta -1 - D\delta^+ ) \eps_0 = - \gamma\, .
\end{equation*}
\item Assume $Y < -\eps_0$. Then the LHS of (\ref{boundgamma}) can be bounded as follows:
\begin{equation*}
( \beta -1 - D\delta^- ) Y + \delta^- \left(D \delta^- - 1\right) \vert \nabla_x V \vert \vert Y^\bot \vert^2   <  -( \beta -1 - D\delta^- ) \eps_0 = - \tilde\gamma\, .
\end{equation*}
\item Assume $\vert Y \vert \leq \eps_0$. Since $1 = \vert Y \vert^2 + \vert Y^\bot \vert^2$, we have $\vert Y^\bot \vert^2 \geq 1 - \eps_0^2$. Further, setting 
$$
h=aY+b \in (\delta^-, \delta^+)\, , \quad
a:=\frac{\delta^+-\delta^-}{2\eps_0}\, , \quad
b:=\frac{\delta^++\delta^-}{2}\, ,
$$
we have $\Gamma'=h$ and $Dh^2-h\leq D \delta^-(\delta^+-1/D)$. Now,
using the fact that the potential gradient is unbounded, we can find a
radius $r_1>0$ large enough such that for all $|x|>r_1$,
\begin{equation*}
  \frac{D(\delta^+ - \delta^-)}{2 \eps_0} -   D
  \delta^-\left(\frac{1}{D}-\delta^+\right)\vert \nabla_x V \vert  < -\frac{2\tilde\gamma}{(1-\eps_0^2)}\, .
\end{equation*}
Putting these estimates together, we obtain for $|x|>r_1$:
\begin{align*}
&(\beta - 1 - Dh) Y 
+\vert Y^\bot \vert^2 \left( \frac{D(\delta^+ - \delta^-)}{2 \eps_0} + \vert \nabla_x V \vert  \left[D h^2 -  h  \right] \right) \\
&\leq (\beta - 1 - D\delta^-) \eps_0 
+\vert Y^\bot \vert^2 \left( \frac{D(\delta^+ - \delta^-)}{2 \eps_0} + \vert \nabla_x V \vert  \left[D \delta^-\left(\delta^+-\frac{1}{D}\right)\right] \right) \\
&\leq \tilde\gamma
+(1-\eps_0^2)\left( \frac{D(\delta^+ - \delta^-)}{2 \eps_0} + \vert \nabla_x V \vert  \left[D \delta^-\left(\delta^+-\frac{1}{D}\right)\right] \right) 
\leq - \tilde\gamma\, .
\end{align*}
\end{itemize}

\bigskip

\noindent
\textbf{\# Step 4: Choice of parameters.}

\bigskip

We now come back to the choice of $\delta^-, \delta^+, \eps_0, \beta$
such that $\kappa \beta < \gamma$ and $0<\gamma<\tilde \gamma$ hold
true. More precisely, these two constraints translate into the
following bound on $\beta$:
\begin{equation}\label{betabound}
 1+D\left(\frac{\delta^++\delta^-}{2}\right)<\beta
<\left(\frac{\eps_0}{\kappa+\eps_0}\right)\left(1+D\delta^+\right) \, .
\end{equation}
It is easy to see that this bound also implies
$1+D\delta^- < \beta < 1+D\delta^+$ as required. However, for this to
be possible we need to choose $\eps_0$ such that $LHS < RHS$, in other
words,
\begin{equation}\label{epsbound}
 \kappa \left( \frac{2+D\left(\delta^++\delta^-\right)}{D\left(\delta^+-\delta^-\right)}\right) < \eps_0\, .
\end{equation}
Since $\eps_0$ has to be less than $1$ and
$D(\delta^+-\delta^-)/(2+D\left(\delta^++\delta^-\right)) < 1/3$,
this bound is only possible if $\kappa \in (0,1/3)$; then it remains
to choose $0<\delta^-<\delta^+<1/D$ such
that
\begin{equation}\label{kappabound}
 \kappa < \frac{D\left(\delta^+-\delta^-\right)}{2+D\left(\delta^++\delta^-\right)} \in \left(0,\frac{1}{3}\right)\, .
\end{equation}
To satisfy all these constraints, we make the choice of parameters (for $\kappa<1/3$):
$$
\delta^+:=\frac{3(1+\kappa)}{4D}, \qquad
\delta^-:=\frac{(1-3\kappa)}{4D}\,  .
$$
Then \eqref{kappabound} holds true, and we can fix $\eps_0 \in (0,1)$ to satisfy \eqref{epsbound}:
\begin{equation*}
 \eps_0:= \frac{1}{2} \left( 1+\kappa\left(\frac{2+D\left(\delta^++\delta^-\right)}{D\left(\delta^+-\delta^-\right)}\right) \right)
 =\frac{1}{2}\left(\frac{1+9\kappa}{1+3\kappa}\right)\, .
\end{equation*}
Finally, we choose $\beta$ satisfying \eqref{betabound} as follows:
\begin{align*}
 \beta&:=\frac12\left[ 1+D\left(\frac{\delta^++\delta^-}{2}\right)
 + \left(\frac{\eps_0}{\kappa+\eps_0}\right)\left(1+D\delta^+\right)
 \right] \\
 &= \frac34+\frac{(1+9\kappa)(7+3\kappa)}{8(6\kappa^2+11\kappa+1)}
 \in (1,2)\, .
\end{align*}

%%%%%%%%%%%%%%%%%%%%%%%%%%%%%%%%%%%%%%%%%%%%%%%%%
%%%%%%%%%%%%%%%%%%%%%%%%%%%%%%%%%%%%%%%%%%%%%%%%%
%%%%%%%%%%%%%%%%%%%%%%%%%%%%%%%%%%%%%%%%%%%%%%%%
%%%%%%%%%%%%%%%%%%%%%%%%%%%%%%%%%%%%%%%%%%%%%%%%%
\section{Existence and uniqueness of a steady state}\label{sec:exuni}
%%%%%%%%%%%%%%%%%%%%%%%%%%%%%%%%%%%%%%%%%%%%%%%%%%%
%%%%%%%%%%%%%%%%%%%%%%%%%%%%%%%%%%%%%%%%%%%%%%%%%
%%%%%%%%%%%%%%%%%%%%%%%%%%%%%%%%%%%%%%%%%%%%%%%%%
%%%%%%%%%%%%%%%%%%%%%%%%%%%%%%%%%%%%%%%%%%%%%%%%
% 
% Let us prove that the linear operator $\op{L}_\kappa$ defined in
% \eqref{kinetic2} is the infinitesimal generator of a $\mC_0$-semigroup
% on $L^2(\rd \mu_\kappa)$.

%%%%%%%%%%%%%%%%%
%%%%%%%%%%%%%%%%%
%%%%%%%%%%%%%%%%%
\subsection{Existence of a $\mC_0$-semigroup}\label{sec:semigroup}
%%%%%%%%%%%%%%%%%%
%%%%%%%%%%%%%%%%%%
%%%%%%%%%%%%%%%%%%
The proof of the next theorem relies on the a priori estimates from
Section \ref{sec:hypo}. 
\begin{theorem}\label{thm:semigroup}
 The linear operator $\op{L}_\kappa : \mD(\op{L}_\kappa) \to L^2(\rd \mu_\kappa)$ defined in \eqref{kinetic2} is the infinitesimal generator of a $\mC_0$-semigroup $(S_t)_{t\geq 0}$ on $L^2(\rd \mu_\kappa)$.
\end{theorem}

\begin{proof}
  Let us denote by $\op{L}_\kappa^*$ the adjoint of $\op{L}_\kappa$ in
  $L^2(\rd \mu_\kappa)$. Both domains $\mD\left(\op{L}_\kappa\right)$
  and $\mD\left(\op{L}_\kappa^*\right)$ contain the core $\mC$ and are
  dense. The operator $\op{L}_\kappa$ is closable in
  $L^2(\rd \mu_\kappa)$. To see this, take a sequence
  $(f_n)_{n \in \N} \in \mD(\op{L}_\kappa)$ converging to zero in
  $L^2(\rd \mu_\kappa)$ such that the sequence
  $(\op{L}_\kappa f_n)_{n \in \N}$ converges to some limit
  $h \in L^2(\rd \mu_\kappa)$. Then for any test function
  $\varphi \in \mC$,
$$
\bla \varphi\, , \, \op{L}_\kappa f_n\bra_{\kappa}
= \bla \op{L}_\kappa^*\varphi\, , \, f_n\bra_{\kappa} \to 0 \qquad \text{as} \, \, n \to \infty\, .
$$
Since the left-hand side converges to
$\bla \varphi\, , \, h\bra_{\kappa}$ for all $\varphi \in \mC$, we
conclude $h\equiv 0$ a.e., and so $\op{L}_\kappa$ is
closable. Similarly, $\op{L}_\kappa^*$ is closable.  We denote by
$\tilde{\op{L}}_\kappa$ and $\tilde{\op{L}}_\kappa^*$ some closed
extensions of $\op{L}_\kappa$ and $\op{L}_\kappa^*$, respectively.
Lumer-Phillips Theorem in the form \cite[Corollary 4.4]{Pazy} states
that an operator $\tilde{\op{L}}$ generates a $\mC_0$-semigroup if
$\tilde{\op{L}}$ is closed and both $\tilde{\op{L}}$ and
$\tilde{\op{L}}^*$ are dissipative.  Since the core $\mC$ is dense in
both $\mD(\tilde{\op{Q}}_\kappa)$ and $\mD(\tilde{\op{Q}}_\kappa^*)$,
which in turn are both dense in $L^2(\rd \mu_\kappa)$, then for any constant $C>0$,
$\tilde{\op{L}}_\kappa-C\op{Id}$ is dissipative if and only if
$\tilde{\op{L}}_\kappa^*-C\op{Id}$ is dissipative.
Therefore, it remains to show that $\tilde{\op{L}}_\kappa-C\op{Id}$ is
dissipative for some $C>0$.  Since the restriction of
$\tilde{\op{L}}_\kappa$ to $\mC$ is $\op{L}_\kappa$, it is enough to
prove that $\op{L}_\kappa-C\op{Id}$ is dissipative on $\mC$ for some constant
$C>0$. The estimates in Section \ref{sec:hypo} show that there exists
$C>0$ s.t. 
\begin{equation*}
\forall f \in \mC, \qquad \bla \op{L}_\kappa f\, , \, f \bra_{\kappa} \leq C \|f\|_{\kappa}^2
\end{equation*}
for some explicit constant $C>0$, which concludes the proof. 
\end{proof}

\subsection{Proof of Theorem \ref{hypothm}}
Proposition \ref{hypoestimate} is the key ingredient to deduce existence of a unique steady state.
The set
\begin{equation*}
 \mB :=\left\{ f \in L^2(\rd \mu_\kappa) \, : \, \op{G}[f] \leq \frac{\gamma_2}{\gamma_1}, \, f \geq 0, \, M_f = 1 \right\}
\end{equation*}
is convex and bounded in $L^2(\rd \mu_\kappa)$ by the norm equivalence
\eqref{normequiv}. By Theorem~\ref{thm:semigroup}, the operator
$\op{L}_\kappa$ generates a $\mC_0$-semigroup $(S_t)_{t\geq 0}$. Then
let us show that $\mB$ is invariant under the action of
$(S_t)_{t\geq 0}$.  Integrating in time the hypocoercivity estimate
\eqref{Gbound} in Proposition \ref{hypoestimate} for any
$f_{\mbox{{\scriptsize in}}} \in L^2(\rd \mu_\kappa)$ with mass $1$, we obtain the bound
\begin{align*}
 \op{G}[f(t)] \leq \op{G}[f_{\mbox{{\scriptsize in}}}] e^{- \gamma_1 t}  + \frac{\gamma_2}{\gamma_1} \left(1-e^{- \gamma_1 t}\right)\, ,
\end{align*}
and thus
\begin{equation*}
 \quad \forall \, t >0, \qquad  \op{G}[f(t)] \leq \max\left\{ \op{G}[f_{\mbox{{\scriptsize in}}}], \frac{\gamma_2}{\gamma_1}  \right\}.
\end{equation*}
Since in addition, $(S_t)_{t\geq 0}$ conserves mass and positivity, we
conclude $S_t(\mB) \subset \mB$ for all times. 

Integrate again the hypocoercivity estimate (\ref{Gbound}) in
Proposition \ref{hypoestimate}, now for the difference of two
solutions with same mass, to get
\begin{align*}
 \op{G}[S_t f - S_t h] \leq  e^{- \gamma_1 t} \op{G}[f - h ]
\end{align*}
for any $t>0$ and $f, h \in \mB$. It follows by Banach's fixed-point theorem that there exists a unique $u^t \in \mB$ such
that $S_t (u^t) = u^t$ for each $t > 0$. % In fact, there exists a
% function $u \in \mB$ such that $S_t (u) = u$ for all
% $t \geq 0$. To see this, 
Let $t_n := 2^{-n}$, $n \in \N$, and $u_n := u^{t_n}$. Then
$S_{2^{-n}} (u_n) = u_n$, and by repeatedly applying the semigroup
property,
\begin{equation}\label{eq:semigroup}
 \forall \, k \in \N, \, \forall \, m \le n\in \N, \quad S_{k2^{-m}} (u_n) = u_n\, .
\end{equation}

Let us prove that $\mB$ is weakly compact in $L^2(\rd
 \mu_\kappa)$. Consider a sequence $(f_n)_{n\in\N} \in \mB$. It has a
 cluster point $f$ for the weak convergence since $\mB$ is bounded in
 $L^2(\rd \mu_\kappa)$, and the corresponding subsequence is still
 denoted $f_n$ for simplicity. By lower semi-continuity of the equivalent
 norm $\op{G}$:
 $$
 \op{G}[f] \leq \liminf_{n \to \infty}\op{G}[f_n] \leq \gamma_2/\gamma_1\, .
 $$
 Further, since $f_n\geq 0$ for all $n \in \N$, it follows that
 $f\geq 0$ (the set of non-negative functions is a strongly closed convex set,
 hence weakly closed). It remains to show that the limit $f$ has mass $1$ by preventing loss of mass at infinity. Use Cauchy-Schwarz's inequality and
 the norm equivalence \eqref{normequiv} to get for $r>0$
 \begin{align*}
 \left(1+\kappa \zeta\right) \left(\int_{|x|>r} \op{\Pi}f_n\, dx\right)^2
%  &\quad
 &\leq
 \left(\int_{|x|>r}\int_{\S^1} f_n^2 \, e^V\frac{\rd x \rd \alpha}{2\pi}\right)\left( \int_{|x|>r}\int_{\S^1} e^{-V}\, \frac{\rd x \rd \alpha}{2\pi}\right)\\
 &\quad+ \kappa \zeta \left(\int_{|x|>r}\int_{\S^1} f_n^2 \, g\frac{\rd x \rd \alpha}{2\pi}\right)\left(\int_{|x|>r}\int_{\S^1} g^{-1}\,\frac{\rd x \rd \alpha}{2\pi}\right)\\
   &\leq
 \|f_n\|_\kappa^2 \left( \int_{|x|>r}\int_{\S^1} \left(e^{-V}+g^{-1}\right) \,\frac{\rd x \rd \alpha}{2\pi}\right)\\
 &\leq
 \left(\frac{2}{1-\eps_1}\right)\frac{\gamma_2}{\gamma_1}\left(\int_{|x|>r}\int_{\S^1} 2e^{-V} \,\frac{\rd x \rd \alpha}{2\pi}\right)\, .
 \end{align*}
 This shows that
 \begin{align*}
  &\sup_{n \in \N} \left(\int_{|x|>r} \op{\Pi}f_n\, dx\right)\\
  &\quad\leq \left(\left(\frac{4}{(1-\eps_1)(1+\kappa \zeta)}\right)
    \left(\frac{\gamma_2}{\gamma_1}\right)\right)^{1/2}
    \left(\int_{|x|>r}\int_{\S^1} e^{-V} \, 
\frac{\rd x \rd \alpha}{2\pi}\right)^{1/2} \to 0 \quad \text{as}\, \, r \to \infty\, ,
 \end{align*}
 % and therefore tightness follows,
 % \begin{equation*}
 % \sup_{n \in \N} \left(\int_{|x|>r} \op{\Pi}f_n\, dx\right) \to 0 \quad \text{as}\, \, r \to \infty\, ,
 % \end{equation*}
 since $\int_{\RR^2\times \S^1} e^{-V} \tfrac{\rd x \rd \alpha}{2\pi}=1$. Together with $M_{f_n}=1$ for all $n \in \N$, it
 follows that $M_f = 1$. Hence
 $f \in \mB$. The weak compactness of $\mB$ implies the existence of a
 subsequence $u_{n_j}$ of $u_n$ and a function $u \in \mB$ such that
 $u_{n_j}$ converges weakly to $u$ in $L^2(\rd \mu_\kappa)$. Letting
 $n_j \to \infty$ in \eqref{eq:semigroup} implies that (since
 $S_t$ is a continuous operator)
 $$
 \forall \, m \in \N, \, \forall \, k \in \N, \quad S_{k2^{-m}} (u) = u\, .
 $$
 % Fix $m \in \N$. Then for all $M \in \N$,
 % $$
 % S_{2^{-m-M}}(u_{m+M}) = u_{m+M}.
 % $$
 % Define $M_j = m_j - m$ for all $ j \in \N$ such that $M_j > 0$. We have
 % $$
 % S_{k 2^{-m}} (u_{m+M_j}) 
 % = S_{k 2^{M_j}2^{-m-M_j}} (u_{m+M_j})
 % = u_{m+M_j}\, .
 % $$
 % By continuity of $S_t(\cdot)$ in the weak topology,
 % $u_{m+M_j} \rightharpoonup u$ as $j \longrightarrow \infty$ implies
 % $$
 % S_{k 2^{-m}} (u_{m+M_j}) \rightharpoonup S_{k 2^{-m}} (u),
 % $$
 % which proves
 % $
 % S_{k 2^{-m}} (u) = u
 % $
 % as claimed.
 % By 
 Finally the density of the dyadic rationals
 $\left\{ k 2^{-m}: k \in \N, m \in \N \right\}$ in $(0,+\infty)$ and
 continuity of $S_t (u)$ in $t$ for all $u \in \mB$ imply that 
 $$
 \forall \, t \geq 0, \quad S_t (u) = u\, .
 $$
 This shows the existence and uniqueness of a global stationary state
 $F_\kappa:=u\in \mB$.\\
 
 To complete the proof of Theorem \ref{hypothm}, we apply the
 hypocoercivity estimate Proposition \ref{hypoestimate} to the difference between a solution $f \in L^2(\rd \mu_\kappa)$ and the unique stationary
 state of the same mass, $M_f F_\kappa$, to show exponential
 convergence to equilibrium in $\|\cdot\|_{\kappa}$: first of all, we deduce from the contraction estimate \eqref{Gbound} that
 \begin{align*}
   \op{G}[f(t)-M_f F_\kappa]  \leq  \op{G}[f_{\mbox{{\scriptsize in}}}-M_f F_\kappa] \, e^{-\gamma_1 t}\, ,
 \end{align*}
which allows then to estimate the difference to equilibrium in the $L^2(\rd \mu_\kappa)$-norm. Indeed, by norm equivalence, we obtain 
\begin{equation*}
\left\| f(t)-M_f F_\kappa \right\|_\kappa^2 \leq \frac{1+\eps_1}{1-\eps_1}\|f_{\mbox{{\scriptsize in}}}-M_f F_\kappa\|_\kappa^2 \, e^{-\gamma_1 t}\, .
 \end{equation*}
Hence, we obtain \eqref{expdecay} with rate of convergence $\lambda_\kappa:= \gamma_1/2$.

\section{Concluding remarks}
\label{sec:concluding-remarks}

From our previous estimates, we have that $\op{G}(F_k)$ is uniformly
bounded in $\kappa$ for $\kappa$ sufficiently small. As a consequence,
$(F_\kappa)_{\kappa >0}$ is a relatively weakly compact family in
$L^2(\rd \mu_\kappa)$, and by uniqueness of the stationary state in
the case $\kappa = 0$, we deduce that $F_\kappa \to F_0$ as
$\kappa \to 0$. It could also be proved with further work that the
optimal (spectral gap) relaxation rate is continuous as
$\kappa \to 0$.

Working in $L^2(\rd \mu_\kappa) \subset L^2(\rd \mu_0)$ we are
treating the operator $\op{L}_\kappa$ as a small perturbation of the
case $\kappa=0$ with stationary conveyor belt. The natural space to
investigate the convergence to $F_\kappa$ in the case $\kappa>0$
however is $L^2\left(F_\kappa^{-1}\, \rd x \, \rd \alpha\right)$.  In
this $L^2$-space the transport operator $\op{T}-\op{P}_\kappa$ is not
skew-symmetric and the collision operator $\op{Q}$ is not
self-adjoint, so the hypocoercivity method \cite{Dolbeault2012long}
cannot be applied. To get around this, one can split the operator
$\op{L}_\kappa$ differently into a transport and a collision part
following the approach in \cite{Calvez}. More precisely, we can write
$\op{L}_\kappa = \tilde{\op{Q}}-\tilde{\op{T}}$ where
  \begin{equation*}
\begin{cases}
   &\tilde{\op{Q}} f = \partial_\alpha \left( D \partial_\alpha f - \frac{\partial_\alpha F_\kappa}{F_\kappa} f\right)\, ,\\[2mm]
   &\tilde{\op{T}} f = \left(\tau + \kappa e_1\right)\cdot \nabla_x f
   - \partial_\alpha\left[ \left(\tau^\bot \cdot \nabla_x V +
       \frac{\partial_\alpha F_\kappa}{F\kappa}\right) f \right]\, .
\end{cases}
  \end{equation*}
  Then in $L^2\left(F_\kappa^{-1}\, \rd x \, \rd \alpha\right)$ the
  operator $\tilde{\op{Q}}$ is symmetric and negative semi-definite,
  and the operator $\tilde{\op{T}}$ is skew-symmetric.  Furthermore,
  the stationary state $F_\kappa$ lies in the intersection of the
  kernels of the collision and transport operators, i.e.
  $F_\kappa \in \text{Ker}\,\tilde{\op{Q}} \cap
  \text{Ker}\,\tilde{\op{T}}$. The hypocoercivity approach requires
  microscopic and macroscopic coercivity of $\tilde{\op{Q}}$ and
  $\tilde{\op{T}}$. Which then requires as in \cite{Calvez} to control
  the behaviour of the stationary state at infinity, i.e. for large
  enough $|x|$,
  \begin{equation*}
   \forall \, \alpha \in \S^1, \quad e^{-\sigma_1 V(x)} \leq F_\kappa (x, \alpha) \leq e^{-\sigma_2 V(x)}
  \end{equation*}
  for some constants $\sigma_1, \sigma_2 >0$. If true, this would be an
  important physical information on the stationary state, but we do
  not know how to prove it at now. Even with this information at hand,
  this approach requires that the existence of the stationary state is
  known a priori. The rate of convergence one obtains in this case may
  be different from the rate obtained here, and it is not clear which
  method yields the better rate as both are most likely not optimal.
  
%%%%%%%%%%%%%%%%%%%%%%%%%%%%%%%%%%%%%%%%
%%%%%%%%%%%%%%%%%%%%%%%%%%%%%%%%%%%%%%%%
%%%%%%%%%%%%%%%%%%%%%%%%%%%%%%%%%%%%%%%%

\section*{Acknowledgments}
EB is very grateful to the University of Cambridge for its sunny
hospitality during the second semester of the academic year
2015-2016. EB and CM acknowledge the support of the ERC Grant MATKIT
(ERC-2011-StG). FH acknowledges support from the Engineering and
Physical Sciences Research Council (UK) grant number EP/H023348/1 for
the University of Cambridge Centre for Doctoral Training, the
Cambridge Centre for Analysis. The authors are very grateful to the
two anonymous referees for very fruitful and detailed comments and
remarks.

%%%%%%%%%%%%%%%%%%%%%%%%%%%%%%%%%%%%%%%%%%%%%%%%%%%%%%%%%%%%%
%%%%%%%%%%%%%%%%%%%%%%%%%%%%%%%%%%%%%%%%%%%%%%%%%%%%%%%%%%%%%
%%%%%%%%%%%%%%%%%%%%%%%%%%%%%%%%%%%%%%%%%%%%%%%%%%%%%%%%%%%%%
%%%%%%%%%%%%%%%%%%%%%%%%%%%%%%%%%%%%%%%%%%%%%%%%%%%%%%%%%%%%%
%%%%%%%%%%%%%%%%%%%%%%%%%%%%%%%%%%%%%%%%%%%%%%%%%%%%%%%%%%%%%
%%%%%%%%%%%%%%%%%%%%%%%%%%%%%%%%%%%%%%%%%%%%%%%%%%%%%%%%%%%%%

\bibliographystyle{abbrv}
\bibliography{hypobooks}

\end{document}